\documentclass[a4paper, notitlepage, 12pt]{article}

\usepackage{ucs}                
\usepackage[utf8x]{inputenc} 	
\usepackage[T1]{fontenc}     
\usepackage{lmodern}         
\usepackage{amsmath}
\usepackage{amssymb}
\usepackage{amsbsy}
\usepackage{amsthm}
\usepackage{wasysym}
\usepackage{float}
\usepackage{algpseudocode}
\usepackage{algorithm}
\usepackage{hyperref}
\usepackage{authblk}
\usepackage[british]{babel}

\usepackage{refdef}
\usepackage{nfontdef}

\newcommand{\ignore}[1]{}

\newcommand{\vek}[1]{\mathchoice{\displaystyle\boldsymbol{#1}}
{\textstyle\boldsymbol{#1}}{\scriptstyle\boldsymbol{#1}}
{\scriptscriptstyle\boldsymbol{#1}}}
\newcommand{\mat}[1]{\mathchoice{\displaystyle\mathbf{#1}}
{\textstyle\mathbf{#1}}{\scriptstyle\mathbf{#1}}
{\scriptscriptstyle\mathbf{#1}}}

\newcommand{\di}{\mathrm{d}}

\usepackage{color}
\definecolor{myred}{rgb}{1, 0.2, 0.2}

\usepackage{multirow}
\usepackage{tikz,pgfplots}
\usetikzlibrary{plotmarks}
\newlength\figureheight
\newlength\figurewidth

\usepackage{graphicx}
\usepackage{caption}
\usepackage{subcaption}
\usepackage{multirow}

\usepackage{pdfpages}
\usepackage{placeins}
\usepackage{enumerate}

\usepackage{amssymb}
\usepackage{amsmath, amsthm}
\usepackage{mathtools}

\DeclareMathOperator{\trace}{trace}

\newtheorem{theorem}{Theorem}[section]

\newtheorem{lemma}[theorem]{Lemma}

\newtheorem{proposition}[theorem]{Proposition}
\newtheorem{remark}[theorem]{Remark}

\newcommand{\spa}[1]{\operatorname{span}#1}
\newcommand{\inner}[2]{\left\langle #1,#2 \right\rangle}
\newcommand{\norm}[1]{\left\|#1\right\|}

\newcommand{\argmin}[1]{\arg\min_{#1}}

\newcommand{\seqt}[2]{\{ #1 \}_{#2}}
\newcommand{\spat}[1]{\operatorname{span}#1}
\newcommand{\innert}[2]{\langle #1,#2 \rangle}
\newcommand{\normt}[1]{\|#1\|}

\newcommand{\bN}{\mathbb{N}}

\newcommand{\bR}{\mathbb{R}}

\newcommand{\cC}{\mathcal{C}}

\newcommand{\cP}{\mathcal{P}}
\newcommand{\cQ}{\mathcal{Q}}

\newcommand{\cS}{\mathcal{S}}

\newcommand{\cU}{\mathcal{U}}

\newcommand{\cX}{\mathcal{X}}

\newcommand\rev[1]{{#1}}

\newcommand{\authorhgm}{\authorcr Hermann G. Matthies}
\newcommand{\authordl}{Dishi Liu}
\newcommand{\authorlg}{Lo\"{i}c Giraldi}
\newcommand{\authoran}{Anthony Nouy}

\newcommand{\affilecn}{\'Ecole Centrale de Nantes, GeM UMR 6183,
  Nantes, France}
\newcommand{\affilwire}{Institute of Scientific Computing, Technische
  Universit\"at Braunschweig, Brunswick, Germany}
\newcommand{\affildlr}{Institute of Aerodynamics and Flow Control,
  German Aerospace Center (DLR), Brunswick, Germany}

\newcommand{\thetitle}{To be or not to be intrusive?\\
The solution of parametric \\
and stochastic equations\\
--- Proper Generalized Decomposition}

\begin{document}

\title{\thetitle\thanks{This work was partly supported by the Deutsche
          Forschungsgemeinschaft (DFG) and by the French National
          Research Agency (Grant ANR CHORUS MONU-0005)}}
\author[a]{\authorlg}
\author[b]{\authordl}
\makeatletter
\author[c]{\authorhgm}
\makeatother
\author[a]{\authoran\thanks{Corresponding author: \texttt{Anthony.Nouy@ec-nantes.fr}}}

\affil[a]{\affilecn}
\affil[b]{\affildlr}
\affil[c]{\affilwire}

\maketitle

\begin{abstract}
  A numerical method is proposed to compute a low-rank Galerkin approximation
  to the solution of a parametric or stochastic equation in a
  non-intrusive fashion. The considered nonlinear problems are
  associated with the minimization of a parameterized differentiable
  convex functional. We first introduce a bilinear parameterization of
  fixed-rank tensors and employ an alternating minimization scheme for
  computing the low-rank approximation. In keeping with the idea of
  non-intrusiveness, at each step of the algorithm the minimizations
  are carried out with a quasi-Newton method to avoid the computation
  of the Hessian. The algorithm is made non-intrusive through the use
  of numerical integration. It only requires the evaluation of
  residuals at specific parameter values. The algorithm is then
  applied to two numerical examples.

  \vspace{5mm}

  \vspace{5mm}
  {\noindent\textbf{Classification:} 65K10, 65D30, 65M70, 15A69, 60H35}
\end{abstract}

\section{Introduction}
\label{sec:introduction}

We are interested in computing the solution of a stochastic parametric
equation. In the literature, methods are said to be non-intrusive when
they require simple calls to the deterministic solver to compute
samples of the solution. We can cite for instance approaches based on
Monte-Carlo, collocation, or $L^2$-projection methods \cite{Nouy2009}.
On the other hand, Galerkin-type methods \cite{Matthies2005} are often
considered as intrusive, as the Galerkin conditions lead to a coupled
system of equations \cite{Giraldi2013,Matthies2005} implying that the
original software for the fixed parameter case can not be used and
requires modification. However, in \cite{Giraldi2013} it was shown
that --- in analogy to the \emph{partitioned} solution of coupled
problems --- it is possible to solve the usual Galerkin equations
non-intrusively in the parametric case by making use of the
``deterministic'' solver, i.e.\ the solver for a fixed value of the
parameters. Recent methods to compute a low-rank approximation
\cite{Chinesta2011,Falco2011a} to such parametric or stochastic
problems also lead to Galerkin-type procedures. Here we want to show
that these methods too can be executed in a non-intrusive manner.

We want to represent the parametric solution $u(p)$ by an
approximation of the form
\begin{equation*}
  u(p) \approx \sum_{i=1}^r \lambda_i(p) v_i,
\end{equation*}
where the $v_i\in\cU$ are fixed vectors and the $\lambda_i(p)$ are real-valued
functions of $p$, and hopefully the \emph{rank} $r$ is sufficiently
small. An obvious advantage of such a
decomposition is the reduction of the number of terms for the
representation of the solution. We also hope to reduce the
computational time for large parametric problems.

The Singular Value Decomposition (SVD) is the best known technique
for constructing a low-rank approximation. If the solution $u$ belongs
to the tensor product of Hilbert spaces, the best low-rank
approximation with respect to the canonical norm is the truncated SVD.
Unfortunately, straightforward computation of the SVD requires to know the
solution of the equation, and thus is not directly
applicable.

An alternative is to use an iterative solver coupled with a low-rank
approximation or truncation technique, leading to \emph{approximate
  iterations}. These methods \cite{Hackbusch2008} have already been
used \cite{Ballani2013,Kressner2011a,Matthies2012} in linear problems,
and could be extended to the iterative solver presented in
\cite{Giraldi2013} in a straightforward manner. 

Another technique called Proper Generalized Decomposition (PGD)
\cite{Chinesta2011,Falco2011a} computes a low-rank approximation of
the solution, relying on a Galerkin-type projection. We distinguish a
progressive and a direct computation of the approximation. The first
one consists in building the approximation in a greedy fashion, with
the computation of a rank one approximation at each iteration, while
the direct approach directly computes a fixed rank approximation in
one go. Such
fixed rank approximation can be computed with an alternating
minimization algorithm in an optimization context \cite{Kolda2009}.

Since the PGD relies on Galerkin-type projections, this method is
usually classified as intrusive. In the present paper, it is shown
that low-rank approximations can be computed in a non-intrusive
fashion by just evaluating residuals. A low-rank approximation is
found by alternating minimization of a convex functional. These
minimizations are carried out with a quasi-Newton technique --- we
choose the quasi-Newton BFGS algorithm here \cite{Dennis1983,Matthies2005} ---
which avoids the computation of the Hessian. As a consequence, the
proposed algorithm only requires evaluations of residuals --- the
negative gradient of the functional --- to compute the low-rank
approximation of the solution. \rev{The efficiency of the proposed approach is essentially related to the number of residual evaluations, which could be reduced by introducing structured approximations of the parameter-dependent residuals. The aim of the present paper is simply to show the feasibility of computing  low-rank Galerkin approximation of the solution in a non-intrusive fashion, based on simple evaluations of the residual for some parameter values.  Efficient implementations will be proposed in a future work.}

The outline of the paper is as follows. In Section
\ref{sec:parametric-problem} the parametric problem is introduced with
a special emphasis on the link between parametric-strong and
parametric-weak formulations. We give necessary conditions for these
problems to be well-posed. Section \ref{sec:basic_pgd} introduces the
different ingredients for computing a low-rank approximation of the
solution with a basic PGD method in a non-intrusive fashion via
numerical integration and the use of a BFGS technique. An improved
algorithm is presented in Section \ref{sec:impr-algor-nonl}. In
Section \ref{sec:numerical-examples}, the method is illustrated with
two numerical examples.

\section{Parametric problems}
\label{sec:parametric-problem}

We consider the parametric problem of finding
$u(p) \in \cU$ such that
\begin{equation}
  \label{eq:param_prob}
  A(u(p);p) = b(p), \quad p \in \cP,
\end{equation}
where $\cU$ is a Hilbert space and $\cP$ is a parameter set equipped
with a finite measure $\mu$ (e.g. probability measure), $A(\cdot;p):
\cU \rightarrow \cU$ and $b(p) \in \cU$. We identify $\cU$
with its dual and we denote by $\innert{\cdot}{\cdot}_{\cU}$ the inner
product on $\cU$ and $\normt{\cdot}_\cU$ the associated norm. A
Galerkin approximation of the solution map could be computed in a non
intrusive manner as in \cite{Giraldi2013}. In this work, we are
interested in finding a low-rank approximation of the solution. To do
so, we assume that Problem \eqref{eq:param_prob} derives from the
minimization of a functional $J(\cdot;p):\cU \ni v \mapsto J(v;p)
\in \bR$. A suitable framework is established with the following
theorem.

\begin{theorem}\label{th:param_framework}
  Assume that
  \begin{enumerate}[(a)]
  \item The map
    $v\mapsto J(v;p)$ is strongly convex uniformly in $p$, meaning
    that there exists a constant $\alpha >0$ independent of $p$ such
    that for all $v,w \in \cU$ and for all $t \in [0,1]$ we have
    $$J(tv+(1-t)w;p) \le t J(v;p) + (1-t) J(w;p) - \frac{\alpha}{2}
    t(1-t) \norm{v-w}_\cU^2,$$
  \item $v\mapsto J(v;p)$ is Fréchet differentiable with gradient
    \begin{equation*}
      \nabla J(v;p) = A(v;p)-b(p),
    \end{equation*}
  \item $p \mapsto A(0;p)-b(p)$ is square $\mu$-integrable,
  \item $p\mapsto J(v;p)$ is $\mu$-integrable and $v \mapsto
    A(v;p)-b(p)$ is Lipschitz uniformly in $p$ on bounded sets,
    meaning that for all bounded sets $\cS \subset \cU$, there exists a constant $K>0$ independent of $p$
    such that
    \begin{equation*}
      \norm{A(v;p)-A(w;p)}_\cU \le K \norm{v-w}_\cU, \quad \forall v,w
      \in \cS.
    \end{equation*}
  \end{enumerate}
  Then a solution of \eqref{eq:param_prob} exists and is unique for
  all parameters $p$ such that we can define a solution map
  $u:\cP\rightarrow \cU$. Moreover, $u$ is in $L^2(\cP;\cU)$, $u$ is
  the unique minimizer in $L^2(\cP;\cU)$ of the functional
  \begin{equation}\label{eq:JP}
    J_\cP:u\mapsto \int_\cP  J(u(p);p)\mu(\di p),
\end{equation}
 and is equivalently characterized by
  \begin{equation}\label{eq:critical_point}
    \int_\cP \inner{A(u(p);p)-b(p)}{\delta u(p)}_\cU \mu(\di p) = 0,
    \quad \forall \delta u \in L^2(\cP;\cU).
  \end{equation}
\end{theorem}
  \begin{proof}
    See Appendix \ref{sec:proof}.
  \end{proof}
It is additionally assumed \cite{Giraldi2013} that an iterative solver
for \eqref{eq:param_prob} is available,
\[u^{(k+1)}(p) \gets u^{(k)}(p) + P^{-1}(R(u^{(k)}(p);p)), \]
convergent for all fixed values of $p$, where 
\[ R(u^{(k)}(p);p) := b(p) - A(u^{(k)}(p);p) \]
is the standard residual of \eqref{eq:param_prob}.
The linear map $P$ is a preconditioner, which may depend on $p$
and on the current iterate $u^{(k)}$; e.g.\ in Newton’s method
$P(u^{(k)};p) = \nabla A(u^{(k)};p)$ --- the Fr\'echet derivative
or gradient of $A(\cdot;p)$.
Usually $P$ is such that $P(u^{(k)};p)(\Delta u^{(k)}) = R(u^{(k)};p)$ 
is ``easy to solve'' for $\Delta u^{(k)} := u^{(k+1)}-u^{(k)}$.
In any case, we assume that for all arguments $p$ and $u$ the map
$v\mapsto P(v;p)$ is linear in $\Delta u$ and non-singular. One should
stress the fact that $P$ or $P^{-1}$ are never needed explicitly, only
their action onto a vector. We assume that the software
interface to the solver for \eqref{eq:param_prob} is such that one may
access the residual $R(u;p)$ without any modification of the software, i.e.\
\emph{non-intrusively}.

In a standard Galerkin method we introduce a finite dimensional space
in $L^2(\cP;\cU)$ where we look for an approximation of the
solution map. Its non-intrusive computation was treated in \cite{Giraldi2013}. Here we are interested in
computing a low-rank approximation of the solution of the form
\begin{equation}\label{eq:lr-form}
  u \approx u_r = \sum_{i=1}^r \lambda_i \otimes v_i, \quad \lambda_i
  \in \cQ, \ v_i \in \cU,
\end{equation}
in a non-intrusive manner, where $\cQ = L^2(\cP)$ is the space of square
$\mu$-integrable functions equipped with its natural inner product
$\innert{\cdot}{\cdot}_\cQ$ and associated norm
$\normt{\cdot}_\cQ$. The search for a low-rank approximation is
justified by the tensor product structure of
$L^2(\cP;\cU)$ which is identified isomorphically with the tensor
Hilbert space $\cQ \otimes \cU$ equipped with the induced canonical norm.

\section{Basic Proper Generalized Decomposition}
\label{sec:basic_pgd}

\subsection{Computation of the approximation}

The basic PGD technique consists in using a greedy rank one
approximation \cite{Falco2011a} for computing an approximation of the
solution of the form \eqref{eq:lr-form}. Assume that we already have
computed $u_r = \sum_{i=1}^r \lambda_i \otimes v_i$ and we want to find an
approximation $u_r + \lambda \otimes v$ of the solution. The couple
$(\lambda,v)$ is computed by solving the minimization problem
\begin{equation*}
  \min_{(\lambda,v)\in \cQ \times \cU} J_\cP(u_r+\lambda \otimes v).
\end{equation*}
The solution is computed using an alternating minimization algorithm.
For $\lambda$ fixed, $v$ is computed solving the Euler-Lagrange equation
related to the minimization with respect to $v$:
\begin{equation*}
  \int_\cP \inner{A(u_r(p) + \lambda(p)v;p)-b(p)}{\lambda(p)
    \delta v}_\cU \mu(\di p) = 0, \quad \forall \delta v \in \cU,
\end{equation*}
or equivalently
\begin{equation}\label{eq:critical_point_v}
   \inner{R_\lambda(v)}{\delta v}_\cU  = 0, \quad \forall \delta v \in \cU,
\end{equation}
with $R_\lambda(v) = \int_\cP \left(b(p)-A(u_r(p) +
  \lambda(p)v;p))\right)\lambda(p)\mu(\di p)$. Similarly, for $v$ fixed,
the minimization on $\lambda$ requires the solution of the nonlinear equation:
\begin{equation*}
  \int_\cP \inner{A(u_r(p) + \lambda(p)v;p)-b(p)}{\delta \lambda(p) v}_\cU \mu(\di p) = 0, \quad \forall \delta \lambda \in \cQ,
\end{equation*}
or equivalently,
 \begin{equation}\label{eq:critical_point_lambda}
  \inner{R_v(\lambda)}{\delta \lambda}_\cQ= \int_\cP R_v(\lambda)(p)  \delta \lambda(p) \mu(\di p) = 0, \quad \forall \delta \lambda \in \cQ,
\end{equation}
with $R_v(\lambda) : p\mapsto \inner{b(p)-A(u_r(p) +
  \lambda(p)v;p)}{v}_\cU$. Problems \eqref{eq:critical_point_v} and
\eqref{eq:critical_point_lambda} are well-defined since they
correspond to the Euler-Lagrange equation related to the minimization
of the strongly convex functionals $J_\lambda:v\mapsto J_\cP(u_r+\lambda \otimes v)$ and $J_v:\lambda \mapsto J_\cP(u_r+\lambda\otimes v)$ respectively. Moreover, the
approximations $\seqt{u_r}{r\in\bN}$ are guaranteed to converge to the
solution (see \cite{Cances2011}).

The equations \eqref{eq:critical_point_v}
and \eqref{eq:critical_point_lambda} can be solved by any suitable
method. For example, Newton's method for \eqref{eq:critical_point_v} is
iterating $v_{k+1} \gets v_k - [\nabla R_\lambda(v_k)]^{-1}
R_\lambda(v_k)$ until convergence, and correspondingly for
\eqref{eq:critical_point_lambda}: $\lambda_{k+1} \gets \lambda_k -
[\nabla R_v(\lambda_k)]^{-1} R_v(\lambda_k)$. Newton's method
can be seen as a prototype algorithm for solving
\eqref{eq:critical_point_v} and \eqref{eq:critical_point_lambda}. The
basic PGD algorithm is summarized in Algorithm \ref{alg:basic_pgd}.

\begin{algorithm}
  \caption{Basic PGD}
  \label{alg:basic_pgd}
  \begin{algorithmic}
    \State  Initialization $u_0$.
    \State $r\gets 0$
    \While {\emph{no convergence of $u_r$}}
    \State Initialize $v$, $\lambda$
    \While  {\emph{no convergence of $\lambda \otimes v$}}
    \State $\lambda \gets \lambda/\normt{\lambda}_\cQ$
    \State Solve Equation \eqref{eq:critical_point_v} for $v$
    \State $v \gets v/\normt{v}_\cU$
    \State Solve Equation \eqref{eq:critical_point_lambda} for $\lambda$
    \EndWhile
    \State $u_{r+1} \gets u_r + \lambda \otimes v$
    \State $r\gets r+1$
    \EndWhile
  \end{algorithmic}
\end{algorithm}

\subsection{Non-intrusive implementation}

\subsubsection{Computation of the projected residuals}
\label{sec:red_res}

To drive to zero the residuals in $R_\lambda(v) = -\nabla J_\lambda(v)$
in \eqref{eq:critical_point_v} and  $R_v(\lambda) = -\nabla J_v(\lambda)$
in \eqref{eq:critical_point_lambda} in the basic PGD
Algorithm \ref{alg:basic_pgd} (e.g.\ by Newton's method as indicated
above in the solutions steps in Algorithm \ref{alg:basic_pgd}), 
those residuals have to be evaluated.  The non-intrusive evaluation
will only use the usual residual $R(v;p)$ of \eqref{eq:param_prob}
as introduced in Section \ref{sec:parametric-problem}.

Let $\{w_z\}$ and $\{p_z\}$ be the weights and points associated with
a quadrature formula on $\C{P}$ for the measure $\mu$. The
residual $R_\lambda(v)$ in expression \eqref{eq:critical_point_v}
then becomes
\[ 
R_\lambda(v) = \int_{\C{P}} \lambda(p)\, R(u_r(p )+\lambda(p )v;p) \; 
\mu(\di p) \approx \sum_z w_z \,\lambda(p_z)\, R(u_r(p_z)+\lambda(p_z)v;p_z) . 
\] 
Similarly, the expression in \eqref{eq:critical_point_lambda} becomes for
all $\delta\lambda\in\C{Q}$
\begin{multline*}
\langle R_v(\lambda),\delta\lambda \rangle_{\C{Q}} = \int_{\C{P}}
\langle R(u_r(p)+\lambda(p) v;p),v \rangle_{\C{U}}\,\delta\lambda(p)\; \mu(\di p) 
\\ \approx \sum_z w_z \, \langle R(u_r(p_z)+\lambda(p_z)v;p_z),
v \rangle_{\C{U}}\, \delta\lambda(p_z) .
\end{multline*}
One may observe from these relations that the computation of 
the residuals in \eqref{eq:critical_point_v} and \eqref{eq:critical_point_lambda}
requires only the evaluation of standard residuals at the quadrature
points $p_z$ of the parametric space with state vector 
$u_r(p_z)+\lambda(p_z)v$, that is
\[
R(u_r(p_z)+\lambda(p_z)v;p_z) = b(p_z) - A(u_r(p_z)+\lambda(p_z)v;p_z).
\]

\subsubsection{Introduction of a quasi-Newton method}
\label{sec:intro_qn}
If one were to use Newton's method for solving \eqref{eq:critical_point_v}
resp.\ \eqref{eq:critical_point_lambda} in Algorithm
\ref{alg:basic_pgd}, one would not only have to evaluate residuals,
but one would also have to evaluate the Hessians of the functionals
(gradients of the residuals). The first Hessian is equal to
      \begin{equation}
\nabla^2 J_\lambda(v) = -\nabla R_\lambda(v) =
\int_{\C{P}} \lambda(p)^2 \, \nabla A(u_r(p) + \lambda(p)v;p)
\;\mu(\di p)\label{eq:H_lambda}
\end{equation}
 and could be called a ``weighted tangent matrix'', and
for the other residual we have for all $\delta\lambda_1,
\delta\lambda_2 \in \C{Q}$:
\begin{multline}\label{eq:H_v} \nabla^2
J_v(\lambda)(\delta\lambda_1,\delta\lambda_2) = -\langle\nabla
R_v(\lambda)\delta\lambda_1,\delta\lambda_2\rangle_\C{Q}\\ : =
\int_{\C{P}} \langle \nabla A(u_r(p) + \lambda(p)v;p)\,v, v
\rangle_{\C{U}} \, \delta\lambda_1(p) \delta\lambda_2(p) \;\mu(\di p).
    \end{multline} Again this would mean accessing the ``tangent
matrix'' $\nabla A(\cdot;p)$, and hence Newton's method cannot be really
carried out non-intrusively. We therefore propose to use a
quasi-Newton method, which only requires evaluation of residuals. This
can be done in a non-intrusive fashion as demonstrated in Subsection \ref{sec:red_res}.

In the following, the symbol
$x$ can stand for $\lambda$ (resp. $v$), $y$ for $v$ (resp. $\lambda$)
and $\cX$ for the space $\cQ$ (resp. $\cU$). The inner product on
$\cX$ is denoted by $\innert{\cdot}{\cdot}_{\cX}$. A quasi-Newton
method \cite{Dennis1983} defines the iterations by
$$x^{(\ell+1)} = x^{(\ell)} + \rho_\ell C_\ell R_y(x^{(\ell)}),$$
where $\rho_\ell$ is a scalar factor to be defined through a
linesearch procedure to be described later, and $C_\ell$ is an
approximation of the inverse of the negative gradient of
$R_y(x^{(\ell)})$ computed with the different iterates of the
algorithm such that the so called quasi-Newton equation $C_{\ell+1}
z_\ell = t_\ell$ is satisfied --- see below for $z_\ell$ and $t_\ell$
--- and the correction of $C_\ell$ in each iteration is of low rank.

Given that we are minimizing a functional, we use here a BFGS method
\cite{Matthies1979} where at iteration $\ell+1$, $C_{\ell+1}$ is
defined recursively by
\begin{equation}\label{eq:sherman-morrison}
  C_{\ell+1} = C_\ell + \frac{\inner{z_\ell}{t_\ell}_\cX +
    \inner{z_\ell}{s_\ell}_\cX}{\inner{z_\ell}{t_\ell}_\cX^2} (t_\ell\otimes t_\ell) -
  \frac{1}{\inner{z_\ell}{t_\ell}_\cX} \left( s_\ell \otimes t_\ell + t_\ell\otimes s_\ell \right).
\end{equation}
where $z_\ell = -(R_y(x^{(\ell+1)})-R_y(x^{(\ell)}))$, $t_\ell =
x^{(\ell+1)} - x^{(\ell)}$, $s_\ell = C_\ell z_\ell$ and $C_0$ is
taken as the formal inverse of a convenient preconditioner to be
defined later.

It should be noted that the algorithm can be performed in a
`matrix-free' formulation, as the matrices $C_\ell$ are only needed
through their action on a vector. Hence they have not to be stored
explicitely \cite{Matthies1979}. 
The application of $C_\ell$ to a vector is described recursively by
\eqref{eq:sherman-morrison}, the action of a typical term, e.g.\ $s_\ell
\otimes t_\ell$, on a vector $x$ being given by $\langle t_\ell, x
\rangle_{\C{X}}\, s_\ell$. The choice of $C_0$ will
be described later. In that way only the vectors $t_\ell$ and $s_\ell$
plus the scalar factors have to be stored for each update. The
application of $C_\ell$ to a vector thus needs two inner products and
a linear combination of three vectors per update. Most often, BFGS is
used in a \emph{limited memory} form \cite{Matthies1979}, with the
number of updates limited to $L$. Once the counter reaches $\ell \ge L$,
either all updates are `forgotten' --- a restart --- or the vectors
$t_\ell$ and $s_\ell$ plus scalar factors are put in a queue of length
$L$, and when the queue is full the first update is popped out and the
last one enqueued; for details see \cite{Matthies1979}.

 With the notations $J_v:\lambda \mapsto
J_\cP(u_r+\lambda\otimes v)$ and $J_\lambda:v \mapsto
J_\cP(u_r+\lambda\otimes v)$, this yields the Algorithm \ref{alg:BFGS}
for computing the solution of $R_y(x) = 0$.
\begin{algorithm}
  \caption{BFGS for computing the solution of $R_y(x) = 0$}
  \label{alg:BFGS}
  \begin{algorithmic}
    \State Initialization of $x^{(0)}$
    \State $C_0 \gets P_y^{-1}\quad (\text{symbolically, inverse of preconditioner}\quad P_y)$
    \State $\ell\gets 0$
    \State $d_0 \gets C_\ell R_y(x^{(0)})$
    \While  {\emph{no convergence}} 
    \State $\rho_\ell \gets$ coarse root of $\rho \mapsto
    \sigma(\rho)$
    \State $t_\ell \gets \rho_\ell d_\ell$
    \State $x^{(\ell+1)} \gets x^{(\ell)} + t_\ell $
    \State $d_{\ell+1} \gets C_\ell R_y(x^{(\ell+1)})$
    \State $s_\ell \gets d_{\ell}-d_{\ell+1}$
    \State Store update information $t_\ell$, $s_\ell$ and scalar factors
    \State $\ell\gets \ell+1$
    \EndWhile 
  \end{algorithmic}
\end{algorithm}

One should bear in mind that this algorithm relies on evaluations of
$R_y(x^{(\ell)})$, that is evaluations of standard residuals according
to Section \ref{sec:red_res}, which makes this algorithm non-intrusive. The scalar $\rho_\ell$ is computed with a coarse
linesearch, which picks $\rho_\ell$ such as to minimize $\rho \mapsto
\varsigma(\rho) := J_y(x^{(\ell)} + \rho\, d_\ell)$. At the minimum we
will have $\sigma(\rho) := \di \varsigma(\rho)/ \di \rho = 0$, which
means $\sigma(\rho) = \langle d_\ell, R_y(x^{(\ell)} + \rho\, d_\ell)
\rangle_{\C{X}} = 0$. The linesearch can thus be carried out by
finding a zero or root of the one-dimensional equation $\sigma(\rho) =
0$, which involves only evaluation of residuals and hence can be
performed non-intrusively. In \cite{Matthies1979} a variant of
\emph{regula falsi} was used for this. The linesearch can be very
coarse, it is in effect an `insurance policy' to avoid divergence in
early iterations. It can be used with Newton's method to increase the
domain of convergence. One may show (see
\cite{Dennis1983,Matthies1979} and the references therein) that as the
method converges, one may choose $\rho_\ell = 1$ so that the
linesearch does not have to be carried out later in the iteration; for
details see \cite{Matthies1979}. Given that $R_y(x)$ can be evaluated
in a non-intrusive fashion with numerical integration, the whole
technique is non-intrusive. The BFGS method is summarized in Algorithm
\ref{alg:BFGS}, and the non-intrusive implementation of the basic PGD
method now uses the BFGS algorithm as described in Algorithm
\ref{alg:BFGS} for the two tasks:
\begin{itemize}
\item Solve equation \eqref{eq:critical_point_v} for $v$.
\item Solve equation \eqref{eq:critical_point_lambda} for $\lambda$.
\end{itemize}



It remains to specify the matrix $C_0 = P_y^{-1}$.  The matrix is only
needed when applied to a vector, thus $P_y^{-1}$ is not needed explicitly.
The preconditioner is best if it is a good approximation of the Hessian
$\nabla^2 J_y$.  Relation \eqref{eq:H_lambda} suggest some very simple choices, e.g.\ for $x = v$ and
$y = \lambda$ when we solve for $v$ (solving equation 
\eqref{eq:critical_point_v}), we may use the original
``deterministic'' preconditioner $P = P(v;p)$ described in 
Section \ref{sec:parametric-problem} to obtain an approximation for 
$\nabla^2 J_\lambda(v)$.  As $\lambda$ is normalized, a very crude
approximation is $P_\lambda := P(u_r(p_a);p_a)$, where $p_a \in\C{P}$ 
is a (possibly well chosen) sample.  This way the preconditioner
is accessible in a non-intrusive fashion.
\begin{remark}
  Another possibility is to directly replace $C_0 R_y(x^{(\ell)})
  \leftrightarrow C_0 R_\lambda(v^{(\ell)})$ (the only context where
  $C_0$ is needed) by
  \begin{multline*}
    P_\lambda^{-1} R_\lambda(v^{(\ell)}) = C_0 R_\lambda(v^{(\ell)}) := \\
    \sum_z w_z (\lambda^{(\ell)}(p_z))^2 \, P^{-1}(u_r(p_z) +
    \lambda^{(\ell)}(p_z)\,v^{(\ell)};p_z) \left( R(u_r(p_z) +
      \lambda^{(\ell)}(p_z)\,v^{(\ell)};p_z) \right) ,
  \end{multline*}
  where each evaluation at a sampling point $p_z$ corresponds to one
  `iteration' of the original deterministic system, a non-intrusive
  computation.
\end{remark}
On the other hand for $x = \lambda$ and $y = v$ when we solve for
$\lambda$ (solving equation \eqref{eq:critical_point_lambda}), we see
from Equation \eqref{eq:H_v} that the action of $\nabla^2
J_v(\lambda)$ is fully diagonalized, it is multiplication by the
positive scalar function $p \mapsto \langle \nabla A(u_r(p) +
\lambda(p)v;p)\,v, v \rangle_{\C{U}}$.
  A very simple choice is replacing that function by a constant $P_v \in \D{R}_+$
  which one may take --- as $v$ is normalised --- inside the convex hull
  of the spectra of the symmetric positive definite operators $\nabla A(v;p)$,
  a crude approximation is the constant function 
  \[ p \mapsto P_v := \langle P(u_r(p_a) + \lambda(p_a)v;p_a) v,
  v\rangle_{\C{U}} \approx \langle \nabla
  A(u_r(p_a)+\lambda(p_a)v;p_a)\,v, v \rangle_{\C{U}} > 0,\] where $p_a
  \in\C{P}$ is again a random (or well chosen) element. The
  application to a function $\lambda \in \C{Q}$ is then the
  multiplication by the \emph{constant} $P_v^{-1} \in \D{R}_+$ defined
  by
\[ P_v^{-1} = \left(\langle  P(u_r(p_a);p_a) v, v \rangle_{\C{U}}\right)^{-1}.\]
 That is certainly
  a non-intrusive computation.  
  \begin{remark}
    Another possibility is to directly replace $C_0 R_y(x^{(\ell)}) =
    C_0 R_v(\lambda^{(\ell)})(p)$ (the only context where $C_0$ is
    needed) for each $p_z$ of the integration rule (the only points in
    $\C{P}$ where it is needed) by
    \begin{multline*}
      P_v^{-1} R_v(\lambda^{(\ell)})(p_z) = C_0 R_v(\lambda^{(\ell)})(p_z) := \\
      \langle P^{-1}(u_r(p_z) + \lambda^{(\ell)}(p_z)\,v^{(\ell)};p_z)
      \left( R(u_r(p_z) + \lambda^{(\ell)}(p_z)\,v^{(\ell)};p_z)
      \right), v^{(\ell)}\rangle_{\C{U}} ,
    \end{multline*}
    where each evaluation at a sampling point $p_z$ corresponds to one
    `iteration' of the original deterministic system at parameter
    value $p_z$ with starting point $u_r(p_z) +
    \lambda^{(\ell)}(p_z)\,v^{(\ell)}$,that is a non-intrusive computation.
  \end{remark}

\section{Improved PGD algorithm}
\label{sec:impr-algor-nonl}

 In the following, we consider that
the cost of the evaluations of $\seqt{u_r(p_z)}{}$ is negligible
compared to the cost of the evaluations of $\seqt{A(u_r(p_z);p_z)-b(p_z)}{}$.
This hypothesis suggests
that the cost of the optimization of all the
$(\lambda_i)_{i=1}^r$, or of all the $(v_i)_{i=1}^r$, should
be almost independent of the rank $r$. We thus propose an improved strategy for computing an approximation of the solution.

\subsection{Low-rank approximation of the solution}
\label{sec:low-rank-appr}

The set of canonical tensors $\cC_r$ of rank at most $r$, defined by
\begin{align*}
  \cC_r = \left\{ \sum_{i=1}^r \lambda_i \otimes v_i;\ \lambda_i
    \in \cQ,\ v_i \in \cU \right\} \subset \cQ \otimes \cU,
\end{align*}
is weakly closed, and the best approximation of a tensor in
$\cC_r$ with respect to the canonical norm is given by the truncated
singular value decomposition (SVD).

A direct low-rank approximation $u_r \in \cC_r$ of the solution $u$
is defined by
\begin{equation}\label{eq:min_cr}
  \min_{v\in\cC_r} J_\cP(v), \quad \text{with} \quad J_\cP(v)=\int_\cP
  J(v(p);p) \mu(\di p),
\end{equation}
where $J$ is defined in Theorem \ref{th:param_framework}.

The set $\cC_r$ is not a vector space, nor a convex set, so that the
computation of the solution to \eqref{eq:min_cr} requires specific algorithms. We introduce a
parameterization $F_r:\cQ^r \times \cU^r \rightarrow \cQ \otimes \cU$
such that $F_r(\cQ^r,\cU^r) = \cC_r$. Let $\vek{\lambda} =
(\lambda_{i})_{1\le i \le r}\in \cQ^r$ and $\vek{v} = (v_{i})_{1\le
  i \le r}\in \cU^r$. The map $F_r$ is defined by
\begin{align}
  \label{eq:param_cr}
  F_r(\vek{\lambda},\vek{v})  = \sum_{i=1}^r \lambda_i \otimes v_i.
\end{align}
Thanks to this parameterization, the problem \eqref{eq:min_cr}
consists in solving
\begin{equation}\label{eq:min_param_tens}
  \min_{\vek{\lambda} \in \cQ^r,\vek{v}\in \cU^r} J_\cP \circ
  F_r(\vek{\lambda},\vek{v}),
\end{equation}

\begin{lemma}\label{lem:bilin_conti}
  The map $F_r$ is bilinear and continuous, such that $F_r$ and its partial maps are Fréchet differentiable.
\end{lemma}
\begin{proof}
  The continuity of $F_r$ comes from the continuity of the tensor
  product $\otimes : \cQ \times \cU \mapsto \cQ \otimes \cU$ with
  respect to the norm $\normt{\cdot}_{\cQ\otimes\cU}$. As a
  consequence, $F_r$ and its partial maps are Fréchet differentiable.
\end{proof}

\begin{remark}
  The representation \eqref{eq:param_cr} is not unique. For $T\in
  GL_r(\bR)$, we denote $T\vek{\lambda} = \seqt{\sum_{j=1}^r T_{ij}
    \lambda_j}{i=1}^r \in \cQ^r$ and $T\vek{v} = \seqt{\sum_{j=1}^r
    T_{ij}v_j}{i=1}^r \in \cU^r$. For all $(\vek{\lambda},\vek{v},T)
  \in \cQ^r \times \cU^r \times GL_r(\bR)$, we have
  $F_r(\vek{\lambda},\vek{v}) = F_r(T\vek{\lambda},T^{-1}\vek{v})$.
  The principal consequence is that there exists an infinite number of
  solutions to the problem \eqref{eq:min_cr}, and that these solutions
  are not isolated. Hence, we can not directly apply a Newton method
  since the Hessian will become ill-conditioned near a critical point.
\end{remark}

\subsection{Adaptive alternating minimization algorithm}
\label{sec:altern-minim-algor}
We solve the problem \eqref{eq:min_param_tens} with an alternating
minimization algorithm, which means that we alternatively solve the
problems 
\begin{equation*}
  \min_{\vek{v}\in
    \cU^r}\int_\cP J(F_r(\vek{\lambda},\vek{v})(p);p)\mu(\di p)
  \quad\text{and} \quad  \min_{\vek{\lambda} \in \cQ^r}\int_\cP
  J(F_r(\vek{\lambda},\vek{v})(p);p)\mu(\di p) 
\end{equation*}
until convergence of $F_r(\vek{\lambda},\vek{v})$. The existence and
the characterization of the solutions of these problems are given in
the following theorem.
\begin{theorem}\label{th:min_altern}
  Under the assumptions of Theorem \ref{th:param_framework}, {if
  $\vek{\lambda}$ is a set of linearly independent functions}, there
  exists a unique solution $\vek{v}\in\cU^r$ to the minimization
  problem
  \begin{equation*}
    \min_{\vek{v} \in \cU^r}\int_\cP
    J(F_r(\vek{\lambda},\vek{v})(p);p)\mu(\di p),
  \end{equation*}
  characterized by the equation
  \begin{equation}\label{eq:stat_v}
    \int_{\cP} \inner{R(F_r(\vek{\lambda},\vek{v})(p);p)}{F_r(
      \vek{\lambda},\delta\vek{v})(p)}_\cU \mu(\di p) = 0, \quad \forall \delta
    \vek{v} \in \cU^r.
  \end{equation}
  Similarly, {if $\vek{v}$ is a set of linearly independent vectors}, there
  exists a unique solution $\vek{\lambda}\in\cQ^r$ to the minimization problem
  \begin{equation*}
    \min_{\vek{\lambda} \in \cQ^r}\int_\cP
    J(F_r(\vek{\lambda},\vek{v})(p);p)\mu(\di p),
  \end{equation*}
  characterized by the equation
  \begin{equation}\label{eq:stat_lambda}
    \int_{\cP} \inner{R(F_r(\vek{\lambda},\vek{v})(p);p)}{F_r(\delta
      \vek{\lambda},\vek{v})(p)}_\cU \mu(\di p) = 0, \quad \forall \delta
    \vek{\lambda} \in \cQ^r.
  \end{equation}
\end{theorem}
\begin{proof}
  See Appendix \ref{sec:proof-theorem-min-alt}.
\end{proof}

We equip the product space $\cU^r$ with the natural inner product
$\innert{\cdot}{\cdot}_{\cU^r}$ defined by
\begin{equation*}
  \inner{\vek{w}}{\vek{v}}_{\cU^r} = \sum_{i=1}^r
  \inner{w_i}{v_i}_{\cU}, \quad \forall \vek{w}
  =(w_i)_{i=1}^r\in\cU^r,\ \forall\vek{v} =(v_i)_{i=1}^r \in \cU^r,
\end{equation*}
and we equip the product space $\cQ^r$ with the natural inner product
$\innert{\cdot}{\cdot}_{\cQ^r}$ defined by
\begin{equation*}
  \inner{\vek{\lambda}}{\vek{\gamma}}_{\cQ^r} = \sum_{i=1}^{r}
  \inner{\lambda_i}{\gamma_i}_\cQ, \quad \forall \vek{\lambda} =
  (\lambda_i)_{i=1}^r \in \cQ^r,\ \forall \vek{\gamma} = (\gamma_i)_{i=1}^r \in\cQ^r.
\end{equation*}

We have to solve the nonlinear Equations \eqref{eq:stat_v} and
\eqref{eq:stat_lambda} that are related to the minimization of some
functionals. Given that
\begin{multline*}
  \int_{\cP}
  \inner{R(F_r(\vek{\lambda},\vek{v})(p);p)}{F_r(\vek{\lambda},\delta
    \vek{v})(p)}_\cU \mu(\di p) \\
    = \sum_{i=1}^r \inner{\int_\cP
      R(F_r(\vek{\lambda},\vek{v})(p);p)
    \lambda_i(p) \mu(\di p)}{\delta v_i}_\cU =
  \inner{R_{\vek{\lambda}}(\vek{v})}{\delta \vek{v}}_{\cU^r},
\end{multline*}
finding the solution to Equation \eqref{eq:stat_v} is equivalent to
finding $\vek{v}$, solution to 
\begin{equation}\label{eq:r_v}
  \inner{R_{\vek{\lambda}}(\vek{v})}{\delta \vek{v}}_{\cU^r} = 0,
  \quad \forall \delta \vek{v} \in \cU^r,
\end{equation}
with $R_{\vek{\lambda}}(\vek{v}) =
(R_{{\lambda}_i}(\vek{v}))_{i=1}^r \in \cU^r$ and
\begin{equation*}
  R_{{\lambda}_i}(\vek{v})=\int_\cP R(F_r(\vek{\lambda},\vek{v})(p);p)
    \lambda_i(p) \mu(\di p), \quad  i \in \{1,\hdots,r\}.
\end{equation*}
It is only necessary to compute $R_{\lambda_i}(\vek{v})(p)$ at
the integration points $p_z$. Once the residuum
$R(F_r(\vek{\lambda},\vek{v})(p_z);p_z)$ has been evaluated it can be
used for all $i \in \{1,\dots,r\}$.

Similarly, using that
\begin{multline*}
  \int_{\cP}
  \inner{R(F_r(\vek{\lambda},\vek{v})(p);p)}{F_r(\delta \vek{\lambda},
    \vek{v})(p)}_\cU \mu(\di p) \\ 
  = \sum_{i=1}^r \int_\cP
  \inner{R(F_r(\vek{\lambda},\vek{v})(p);p)}{v_i}_\cU \delta
  \lambda_i(p) \mu(\di p) = \inner{R_{\vek{v}}(\vek{\lambda})}{\delta
    \vek{\lambda}}_{\cQ^r},
\end{multline*}
finding the solution to Equation \eqref{eq:stat_lambda} is equivalent to
finding $\vek{\lambda}$, solution to 
\begin{equation}\label{eq:r_lambda}
  \inner{R_{\vek{v}}(\vek{\lambda})}{\delta \vek{\lambda}}_{\cQ^r} =
  0, \quad \forall \delta \vek{\lambda} \in \cQ^r,
\end{equation}
with $R_{\vek{v}}(\vek{\lambda}) = (R_{v_i}(\vek{\lambda}))_{i=1}^r \in
\cQ^r$ and
\begin{equation*}
  R_{v_i}(\vek{\lambda}): p\mapsto
  \inner{R(F_r(\vek{\lambda},\vek{v})(p);p)}{v_i}_\cU, \quad i \in
  \{1,\hdots,r\},
\end{equation*}
approximated again by standard residuum evaluations;
and $R(F_r(\vek{\lambda},\vek{v})(p_z);p_z)$ has to be evaluated only
once for all $i \in \{1,\dots,r\}$. 

We can thus use once again a BFGS method to solve Problems
\eqref{eq:r_v} and \eqref{eq:r_lambda}. Moreover, the algorithm can be
made non-intrusive using numerical integration given that
\begin{align*}
  \inner{R_{v_i}(\vek{\lambda})}{\delta \lambda_i}_{\cQ} \approx
  \sum_z w_z \inner{R(F_r(\vek{\lambda},\vek{v})(p_z);p_z)}{v_i \delta
  \lambda_i(p_z)}_\cU
\end{align*}
and
\begin{equation*}
  R_{\lambda_i}(\vek{v})\approx \sum_z w_z R(F_r(\vek{\lambda},\vek{v})(p_z);p_z)
  \lambda_i(p_z).
\end{equation*}
With $u_r = F_r(\vek{\lambda},\vek{v})$, we insist on the fact that a
BFGS technique will require the evaluations of the residual
$R(u_r(p_z);p_z) = b(p_z)-A(u_r(p_z);p_z)$.

In order to avoid any degeneracy and obtain well-conditioned problems, we introduce two orthogonalization
steps. We denote by $orth:\cX^r \rightarrow \cX^r$, $\cX = \cQ$ or
$\cU$, an operator such that with $\vek{x}' = orth(\vek{x})$, we have
$\spa \vek{x} \subset \spa \vek{x}'$, and $\vek{x}' = (x_i')_{i=1}^r$ is
an orthonormal set. Such a set can be obtained by taking the first $r$ left
singular vectors of $\vek{x}$ considered as a tensor in $\cX \otimes
\bR^r$ for instance.

Finally, the rank is adapted by choosing a good initial guess at each
step. Except for the rank one approximation, the initial guess for the
computation of the rank $r$ approximation is chosen to be the rank
$r-1$ approximation of the solution computed at the previous iteration
plus a rank one term. The whole approach is summarized in Algorithm
\ref{algo:ni_aama}.

\begin{algorithm}
  \caption{Non-intrusive implementation of the improved PGD}
  \label{algo:ni_aama}
  \begin{algorithmic}
    \State  Initialization of $u_0$.
    \State $r\gets 1$
    \While {\emph{no convergence of $u_r$}}
    \State Initialize $v_r$, $\lambda_r$
    \State $\vek{\lambda} \gets (\lambda_i)_{i=1}^r$
    \State $\vek{v} \gets (v_i)_{i=1}^r$
    \While  {\emph{no convergence of $F_r(\vek{\lambda},\vek{v})$}}
    \State $\vek{\lambda} \gets orth(\vek{\lambda})$
    \State Solve Equation \eqref{eq:r_v} using Algorithm \ref{alg:BFGS}
    \State $\vek{v} \gets orth(\vek{v})$
    \State Solve Equation \eqref{eq:r_lambda} using Algorithm \ref{alg:BFGS}
    \EndWhile
    \State $u_{r} \gets F_r(\vek{\lambda},\vek{v})$
    \State $r\gets r+1$
    \EndWhile
  \end{algorithmic}
\end{algorithm}

The operator $C_0$ has now to be defined. Again, we propose
a priori good approximations $C_0^{-1}$ of the Hessian of the functional in order
to improve the performance of the BFGS method. We observe that the Hessian
$H_{\vek{\lambda}}(\vek{v})$ of $J_\cP \circ
F_r(\vek{\lambda},\vek{v})$ where the derivative is taken with respect to $\vek{v}$ is
\begin{equation*}
  \inner{H_{\vek{\lambda}}(\vek{v}) \delta \vek{v}}{\delta
    \vek{v}'}_{\cU^r} = \sum_{i=1}^r \sum_{j=1}^r \inner{\left(\int_\cP\lambda_i(p)\lambda_j(p)
    \nabla A(u_r(p);p)\mu(\di p)\right) \delta v_{i}}{\delta v_{j}'}_\cU.
\end{equation*}
Given that $\vek{\lambda}$ is a family of orthonormal vectors, this
suggests that $C_0^{-1}$ could be approximated by a block-diagonal version of
the preconditioner proposed in Section \ref{sec:intro_qn}, each block
being defined by $P(u_r(p_a);p_a)$.

\begin{remark}
  Similarly, the Hessian $H_{\vek{v}}(\vek{\lambda})$ of $J_\cP \circ
  F_r(\vek{\lambda},\vek{v})$ where the derivative is taken with
  respect to $\vek{\lambda}$ is
  \begin{equation*}
    \inner{H_{\vek{v}}(\vek{\lambda})\delta \vek{\lambda}}{\delta\vek{\lambda}'}_{\cQ^r} = \sum_{i=1}^r \sum_{j=1}^r \int_\cP
    \inner{\nabla A(u_r(p);p)v_i}{v_j}_\cU \delta{\lambda}_{i}(p)\delta {\lambda}_{j}' \mu(\di p).
  \end{equation*}
  Again, a simple approximation of the Hessian is a block diagonal
  version of the preconditioner proposed in Section
  \ref{sec:intro_qn}, the $i$\textsuperscript{th} block being defined
  by $\alpha_i \mathrm{Id}_{\cU}$ with $\alpha_i =
  \inner{P(u_r(p_a);p_a)v_i}{v_i}_\cU$. Note that all proposed
  approximations of the Hessian require the computation of only one
  preconditioner at the parameter $p_a$.
\end{remark}

\subsection{Finite dimensional case - Algebraic form}
\label{sec:algebraic-context}

We assume that $\cU$ is a finite dimensional vector space. Let
$\{e_i\}_{i=1}^n$ be a basis of $\cU$. We denote by $\cQ_m =
\spat{\{\psi_j\}_{j=1}^m} \subset \cQ$, where $\{\psi_j\}_{j=1}^m$
is a basis of $\cQ_m$. We introduce thus a finite dimensional space
$\cQ_m \otimes \cU \subset \cQ \otimes \cU$ for approximating the
solution.

A tensor $u \in \cQ_m \otimes \cU$ can thus be written
\begin{align*}
  u = \sum_{i=1}^m \sum_{j=1}^n u_{ij} \psi_i
  \otimes e_j,
\end{align*}
and a low-rank tensor is given as
\begin{align*}
  u_r = F_r(\vek{\lambda},\vek{v}) = \sum_{k=1}^r \lambda_k \otimes v_k =
  \sum_{i=1}^m \sum_{j=1}^n \left(\sum_{k=1}^r
    \lambda_{ik} v_{jk} \right)\psi_i \otimes e_j.
\end{align*}
We denote by $\mat{\Lambda} \in \bR^{m\times r}$, $\mat{V} \in
\bR^{n \times r}$ and $\mat{F}_r(\mat{\Lambda},\mat{V}) \in \bR^{m
\times n}$ matrices such that
\begin{align*}
  \mat{\Lambda}_{ik} = \lambda_{ik}, \quad  \mat{V}_{j k}
  = v_{j k} \quad \text{and} \quad \mat{F}_r(\mat{\Lambda},\mat{V}) = \mat{\Lambda} \mat{V}^T.
\end{align*}
With these notations, a low-rank tensor can be expressed under the
different forms
\begin{align*}
  u_r = F_r(\vek{\lambda},\vek{v}) = \sum_{k=1}^r \lambda_k \otimes v_k =
  \sum_{i=1}^m \sum_{j=1}^n
  \mat{F}_r(\mat{\Lambda},\mat{V})_{i j} \psi_i \otimes e_j,
\end{align*}
thus allowing the identification of $F_r$ with $\mat{F}_r:\bR^{m\times r}
\times \bR^{n \times r} \rightarrow \bR^{m \times n}$. Denoting by
$\mat{J}_\cP:\bR^{m \times n}\rightarrow \bR$ the functional such that
$J_\cP \circ F_r(\vek{\lambda},\vek{v}) = \mat{J}_\cP \circ
\mat{F}_r(\mat{\Lambda},\mat{V})$, we can consider the minimization
problem, equivalent to Problem \eqref{eq:min_param_tens}, defined
by
\begin{align*}
  \min_{\mat{\Lambda} \in \bR^{m\times r},\mat{V} \in \bR^{n \times
      r}} \mat{J}_\cP \circ \mat{F}_r(\mat{\Lambda},\mat{V}),
\end{align*}
and directly use all the algorithms described in Section
\ref{sec:basic_pgd} and \ref{sec:impr-algor-nonl} in the general setting.

We denote by $\vek{\psi}_z$ the vector of evaluations of the basis
functions of $\cQ_m \subset \cQ = L^2(\cP)$ at the parameter value
$p_z \in \cP$, defined by $\vek{\psi}_z =
(\psi_i(p_z))_{i=1}^m \in \bR^{m}$. 

We set $u_r(p_z) = F_r(\vek{\lambda},\vek{v})(p_z)$ and denote by
$\vek{R}(u_r(p_z);p_z) \in \bR^n$ the vector defined by
$$(\vek{R}(u_r(p_z);p_z))_j =
\inner{R(u_r(p_z);p_z)}{e_j}_{\cU} = \inner{b(p_z)-A(u_r(p_z);p_z)}{e_j}_{\cU}, \quad \forall j\in\{1,\hdots,n\}.$$
We deduce the algebraic form of \eqref{eq:r_v} and
\eqref{eq:r_lambda}, that is
\begin{align*}
  \inner{R_{\vek{\lambda}}(\vek{v})}{\delta \vek{v}}_{\cU^r} \approx& \sum_z w_z
  \vek{\psi}_z^T \mat{\Lambda} \delta \mat{V}^T
  \mat{R}(u_r(p_z);p_z), \\=& \inner{\left(\sum_z w_z \mat{R}(u_r(p_z);p_z)
      \vek{\psi}_z^T\right) \mat{\Lambda}}{\delta \mat{V}}_{\bR^{n\times r}},
\end{align*}
and
\begin{align*}
  \inner{R_{\vek{v}}(\vek{\lambda})}{\delta \vek{\lambda}}_{\cQ^r}
  \approx& \sum_z w_z \vek{\psi}_z^T \delta \mat{\Lambda} \mat{V}^T
  \mat{R}(u_r(p_z);p_z), \\
  &= \inner{\left( \sum_z w_z \vek{\psi}_z \mat{R}(u_r(p_z);p_z)^T \right)
    \mat{V}}{\delta \mat{\Lambda}}_{\bR^{m\times r}}
\end{align*}
where $\innert{\mat{X}}{\mat{Y}}_{\bR^{s\times t}} =
\trace(\mat{X}\mat{Y}^T)$ is the canonical inner product in the matrix
space $\bR^{s\times t}$. This clearly shows that
$R_{\vek{\lambda}}(\vek{v})$ and $R_{\vek{v}}(\vek{\lambda})$ can be
evaluated in a non-intrusive fashion thanks to simple evaluations of
the original residual $R(u_r(p_z);p_z)=b(p_z)-A(u_r(p_z);p_z)$ defined
in Section \ref{sec:parametric-problem} and hence is a non-intrusive computation.

\section{Numerical examples}
\label{sec:numerical-examples}

\subsection{Electronic network}
\label{sec:electronic-network}

\renewcommand{\mathbf}[1]{#1}

In this section, we use the example introduced in \cite{Giraldi2013}.
It is a simple electronic network. The original equation to be
solved is
\begin{equation*}
  \mathbf{B} \mathbf{u}(\mathbf{p}) + (p_1 + 2) (\mathbf{u}(\mathbf{p})^T \mathbf{u}(\mathbf{p}))
  \mathbf{u}(\mathbf{p}) = (p_2 + 25) \mathbf{f},
\end{equation*}
with
\begin{equation*}
  \mathbf{B} = \frac{1}{R}
  \begin{pmatrix*}[r]
    3 & -1 & -1 & 0 & -1 \\
    -1 & 3 & -1 & -1 & 0 \\
    -1 & -1 & 4 & -1 & -1\\
    0 & -1 & -1 & 3 & -1 \\
    -1 & 0 & -1 & -1 & 4
  \end{pmatrix*}, \quad
  \mathbf{f} =
  \begin{pmatrix}
    1 \\ 0 \\ 0 \\ 0 \\ 0
  \end{pmatrix} \quad \text{and} \quad R = 100,
\end{equation*}
where the matrix $\mathbf{B}$ represents the network from Figure
\ref{fig:elec_net}, and $p=(p_1,p_2)$ where $p_1$ and $p_2$ are
uniform random variables on $[-1,1]$. The matrix has this simple form
as we have chosen all resistors equal.
\begin{figure}[h]
  \centering
  \includegraphics[height=5cm]{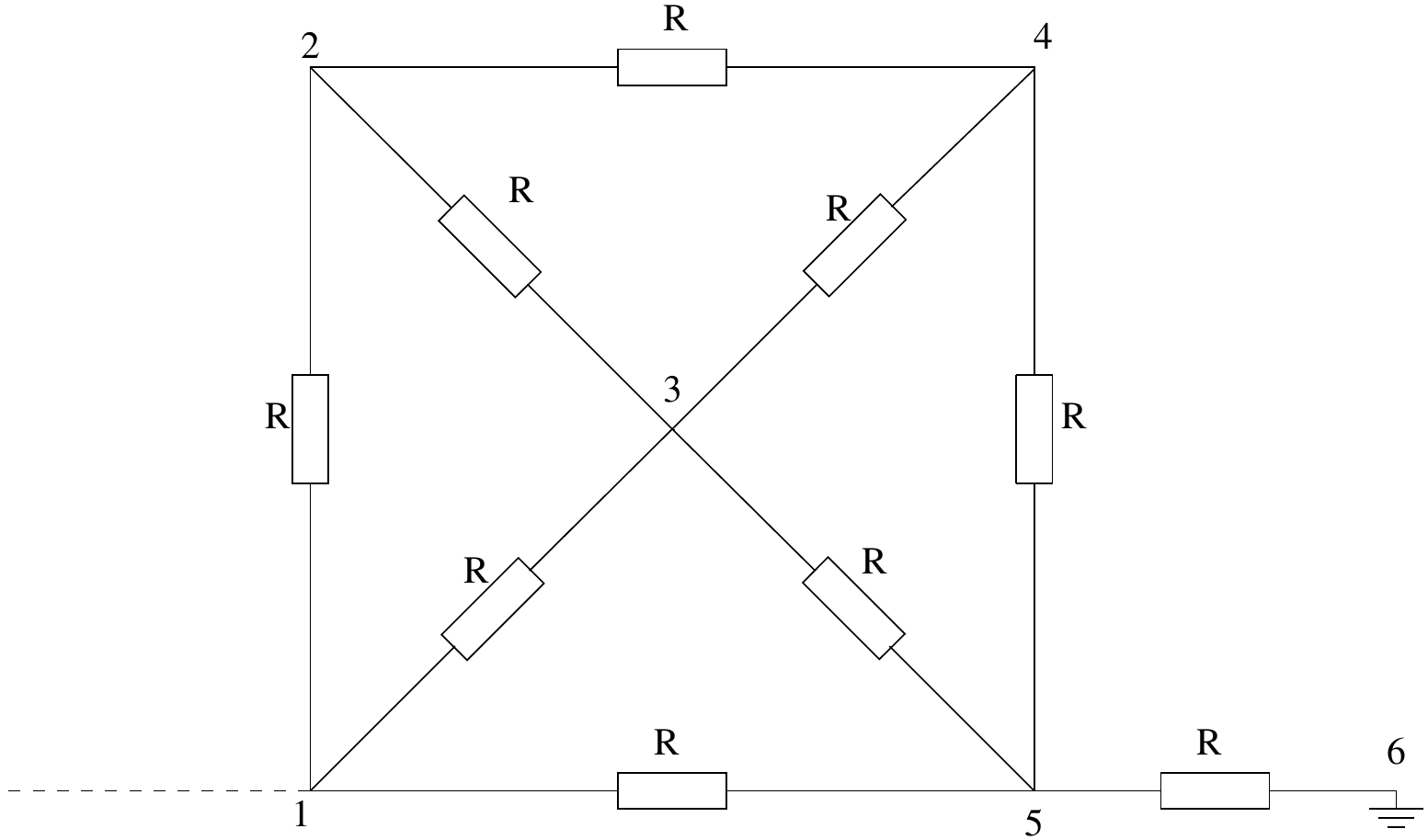}
  \caption{Electronic network.}
  \label{fig:elec_net}
\end{figure}
The problem is related to the minimization of the functional defined
by
$$J(v;p) =
\frac{1}{2} \mathbf{v}^T\mathbf{B} \mathbf{v}+ \frac{1}{4}(p_1 + 2) (\mathbf{v}^T
\mathbf{v})^2 - (p_2 + 25) \mathbf{v}^T\mathbf{f}.$$
The residual $\mathbf{R}(u(p);p) = -\nabla J(u(p);p)$ is thus given by
\begin{equation*}
  \mathbf{R}(\mathbf{u}(\mathbf{p});p) = (p_2 + 25) \mathbf{f} - \left(\mathbf{B}
  \mathbf{u}(\mathbf{p}) + (p_1 + 2) (\mathbf{u}(\mathbf{p})^T \mathbf{u}(\mathbf{p}))
  \mathbf{u}(\mathbf{p})\right).
\end{equation*}
\begin{proposition}\label{prop:elec_network}
  The assumptions of Theorem \ref{th:param_framework} are satisfied.
\end{proposition}
\begin{proof}
  See Appendix \ref{sec:proof_prop_elec_network}.
\end{proof}

The basis of the finite dimensional stochastic space $\cQ_m\subset
\cQ$ used for the approximation is chosen to be the multidimensional
Legendre polynomials $\seqt{\psi_j}{j=1}^m$ of total degree $d$ which
are orthogonal for the measure $\mu$. For the quadrature rule, we have
chosen a full tensorization of unidimensional Gauss-Legendre
quadrature with $(d+1)$ points, such that
the total number of quadrature points is $(d+1)^2$.

We can now directly apply the basic PGD procedure shown in Algorithm
\ref{alg:basic_pgd} and the improved algorithm
described in Algorithm \ref{algo:ni_aama} in order to find a low-rank
approximation of the solution map $\mathbf{u}$. The convergence of the
algorithms is controlled by a stagnation criterion. 

Concerning the basic PGD, the stagnation criterion is set to
  $10^{-2}$ with a maximum of $10$ iterations for each alternating
  minimization algorithm. The tolerance of the BFGS method is set to
  $10^{-10}$. For Problem \eqref{eq:critical_point_lambda}, the
  initialization of the preconditioner for the BFGS algorithm is the
  identity, and $B$ for Problem \eqref{eq:critical_point_v}.
  Moreover $\lambda$ is initialized with a vector full of ones and $v$
  with a vector full of $10^{-8}$. Being close to 0 is beneficial for
  later iterations, as the corrections will only slightly improve the
  approximation. However we observed that initializing $v$ to 0 may
  induce that the next problem on $\lambda$ becomes ill-conditioned
  due to the equality $\lambda \otimes 0 = 0$ for all $\lambda$. For
  the improved PGD, the stagnation criterion is set to
  $\max(10^{-(r+1)},10^{-8})$ with a maximum of $20$ iterations for
  the alternating minimization algorithm. Both algorithms are initialized with $u_0 = 0$.

The relative error is measured with respect to the norm
\begin{align}\label{eq:err_estim}
  \varepsilon(u_r) = \frac{\norm{u-u_r}_{\cQ\otimes\cU}}{\norm{u}_{\cQ\otimes\cU}} &= \sqrt{\frac{\int_\cP \norm{u(p)-u_r(p)}_\cU^2
  \mu(dp)}{\int_\cP \norm{u(p)}_\cU^2
  \mu(dp)}} \\ &\approx \sqrt{\frac{\sum_z w_z
  \norm{u(p_z)-u_r(p_z)}_\cU^2}{\sum_z w_z \norm{u(p_z)}_\cU^2}}, \nonumber
\end{align}
with $u$ the exact solution, $u_r$ the low-rank approximation and
using a fully tensorized Gauss-Legendre quadrature with a number of
points $20^2=400$. The deterministic solutions
$\seqt{\mathbf{u}(\mathbf{p}_z)}{}$ are computed using a modified
Newton algorithm where the tangent matrix is chosen to be the linear
part $\mathbf{B}$ of the Hessian of the functional $J$. The low-rank
approximations are also compared to the full-rank Galerkin
approximation computed with the block-Jacobi algorithm introduced in
\cite{Giraldi2013}, with a stagnation criterion of $10^{-10}$. The
comparison is made in Table \ref{table:elec_error} for total degrees $d =\ $2,3,4,5 and ranks
1,2,3,4,5 for the approximations.

\begin{table}[htbp]
\begin{center}
 \begin{tabular}{|c|c|c|c|c|}
   \hline & $d=2$ & $d=3$ & $d=4$ & $d=5$ \\ 
   \hline \hline \multicolumn{5}{|c|}{Block-Jacobi solver \cite{Giraldi2013}} \\ \hline
   & $5.14 \times 10^{-5}$ & $3.31 \times 10^{-6}$&   $2.31\times
   10^{-7}$ & $1.70\times 10^{-8}$ \\ \hline \hline
   \multicolumn{5}{|c|}{Basic PGD (Algorithm \ref{alg:basic_pgd})} \\ \hline
   $r=1$ & $2.34\times 10^{-3} $& $2.34\times 10^{-3} $&  $2.34\times 10^{-3}$&$  2.34\times 10^{-3} $\\ \hline
   $r=2$ & $9.67\times 10^{-5} $& $8.22\times 10^{-5} $&  $8.22\times 10^{-5}$&$  8.22\times 10^{-5} $\\ \hline
   $r=3$ & $5.14\times 10^{-5} $& $3.39\times 10^{-6} $&  $8.03\times 10^{-7}$&$  7.78\times 10^{-7} $\\ \hline
   $r=4$ & $5.14\times 10^{-5} $& $3.31\times 10^{-6} $&  $2.34\times 10^{-7}$&$  3.63\times 10^{-8} $\\ \hline
   $r=5$ & $5.14\times 10^{-5} $& $3.31\times 10^{-6} $&  $2.31\times 10^{-7}$&$  1.71\times 10^{-8} $\\
   \hline \hline
   \multicolumn{5}{|c|}{Improved PGD (Algorithm \ref{algo:ni_aama})} \\ \hline
   $r=1 $ & $ 2.34\times 10^{-3} $ & $  2.34\times 10^{-3} $ & $  2.34\times 10^{-3} $ & $  2.34\times 10^{-3} $\\ \hline
   $r=2 $ & $ 5.14\times 10^{-5} $ & $  3.31\times 10^{-6} $ & $  2.85\times 10^{-7} $ & $  1.95\times 10^{-7} $\\ \hline
   $r=3 $ & $ 5.14\times 10^{-5} $ & $  3.31\times 10^{-6} $ & $  2.31\times 10^{-7} $ & $  1.79\times 10^{-8} $\\ \hline
   $r=4 $ & $ 5.14\times 10^{-5} $ & $  3.31\times 10^{-6} $ & $  2.31\times 10^{-7} $ & $  1.79\times 10^{-8} $\\ \hline
   $r=5 $ & $ 5.14\times 10^{-5} $ & $  3.31\times 10^{-6} $ & $  2.31\times 10^{-7} $ & $  1.76\times 10^{-8} $\\ \hline
 \end{tabular}
\end{center}
\caption{Relative error for the approximation resulting from the
  block-Jacobi solver, the basic PGD and the improved algorithm for different total degrees $d$ and different
  $r$.}
\label{table:elec_error} 
\end{table}







We can observe from Table \ref{table:elec_error} that the low-rank
approximation gives a good approximation of the solution, even with a
rank one approximation. Moreover, for a rank greater than 3, the
low-rank approximation always gives results as good as the one of the
full-rank Galerkin approximation. Besides, in this example, we can
see that the greedy approximation gives satisfying results, even if
the result is not optimal compared to the approximation resulting from
a direct optimization in low-rank subsets.

For the rest of this section, we focus on $d=5$ and we measure the
efficiency of the different algorithms by counting the number of calls
to the residual $R(u_r(p_z);p_z) = b(p_z)-A(u(p_z);p_z)$. The results are
reported in Table \ref{table:elec_perf}.
\begin{table}[htbp]
\begin{center}
 \begin{tabular}{|c|c|c|c|c|c|}
   \hline & $r=1$ & $r=2$ & $r=3$ & $r=4$ & $r=5$ \\ \hline \hline
   \multicolumn{6}{|c|}{Basic PGD (Algorithm \ref{alg:basic_pgd})} \\ \hline
   Relative error & $  2.34\times 10^{-3} $ &$  8.22\times 10^{-5} $&$
   7.78\times 10^{-7} $&$  3.63\times 10^{-8} $&$  1.71\times 10^{-8}
   $\\\hline
   Residual calls &1044 &2160 &3096 &3816 &4464 \\ \hline \hline
   \multicolumn{6}{|c|}{Improved algorithm
     (Algorithm \ref{algo:ni_aama})} \\ \hline
   Relative error & $  2.34\times 10^{-3} $& $  1.95\times 10^{-7} $&
   $  1.79\times 10^{-8} $ & $  1.79\times 10^{-8} $&$  1.79\times 10^{-8} $ \\ \hline
   Residual calls &1044 &2304 &2700 &2844 &3024 \\ \hline
 \end{tabular}
\end{center}
\caption{Number of calls to the residual and corresponding relative
  error for different ranks $r$ for the basic PGD and the improved algorithm.}
\label{table:elec_perf} 
\end{table}

Both algorithms are similar at the beginning until $r=2$. When $r=3$,
Algorithm \ref{algo:ni_aama} becomes more efficient for computing
the low-rank approximation. However, if we compare with the
block-Jacobi solver, the latter one only requires $540$ calls to the residual.
This suggests that the classical algorithms for computing the low-rank
approximation of the solution of nonlinear equations must be
reconsidered in terms of efficiency and intrusivity and different
approaches must be proposed.

\subsection{Obstacle problem}
\label{sec:obstacle_problem}

We consider the obstacle problem introduced in
\cite{Cances2011}. A rope is clamped at its extremities over an
obstacle modeled by a function $g$, and a force $f$ is applied to the
rope. Noting $\Omega = (0,1)$, the vertical displacement of the rope
is modeled by the function $v:\Omega \times \cP \rightarrow \bR$. The force is set to be
constant $f \equiv 1$, and the obstacle is defined by 
\begin{align*}
  g(p;x) = p [\sin(3\pi x)]_+ + (p-1) [\sin(3\pi x)]_-, \quad \forall
  (x,p) \in \Omega \times \cP,
\end{align*}
where $p$ is uniformly distributed in $\cP = (0,1)$. The non
penetration condition (i.e. $v \ge g$) is taken into account with
a penalty formulation, with the penalty coefficient $\rho = 10^3$. The
reference solution $u$ can be found by solving
\begin{equation*}
  \min_{v \in L^2(\cP) \otimes H_0^1(\Omega)} J_\cP(v),
\end{equation*}
with
\begin{align*}
  J_\cP(v) = \int_\cP J(v(p);p)~\mu(\di p)
\end{align*}
and
\begin{align*}
  J(v;p) = \int_\Omega \frac{1}{2} \left(\frac{\partial}{\partial x}
      v(x)\right)^2 - f(p;x) v(x)+ \frac{\rho}{2} [v(x) - g(p;x)]_+^2~\di x.
\end{align*}
Following the proofs in \cite{Cances2011}, we can show that the assumptions of Theorem \ref{th:param_framework} are satisfied.
\begin{figure}[h]
  \centering
  \begin{subfigure}[b]{0.4\textwidth}
    \includegraphics[width=\textwidth]{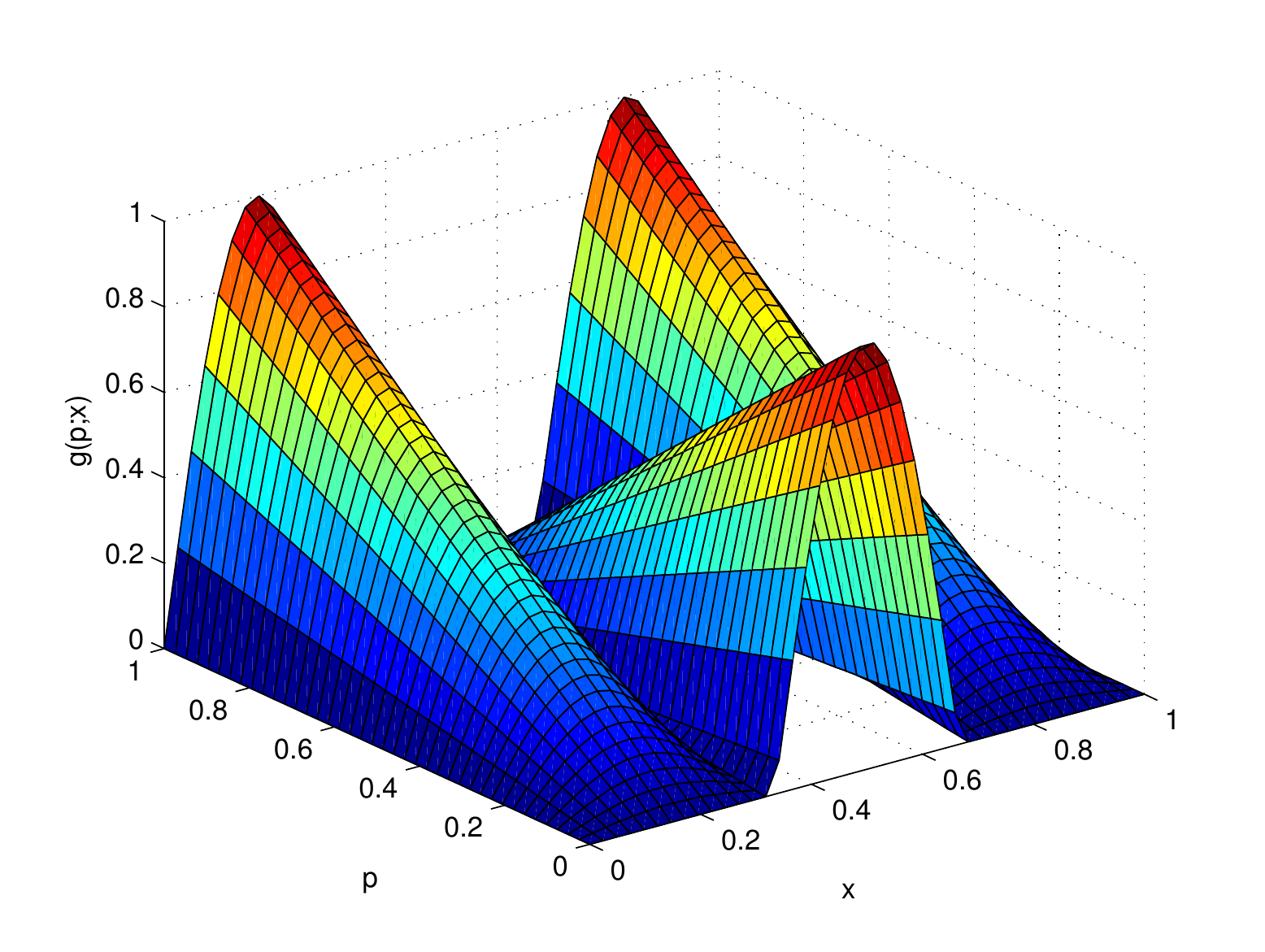}
    \caption{Obstacle: $g(p;x)$.}
    \label{fig:altitude_obstacle}
  \end{subfigure}
  ~
  \begin{subfigure}[b]{0.4\textwidth}
    \includegraphics[width=\textwidth]{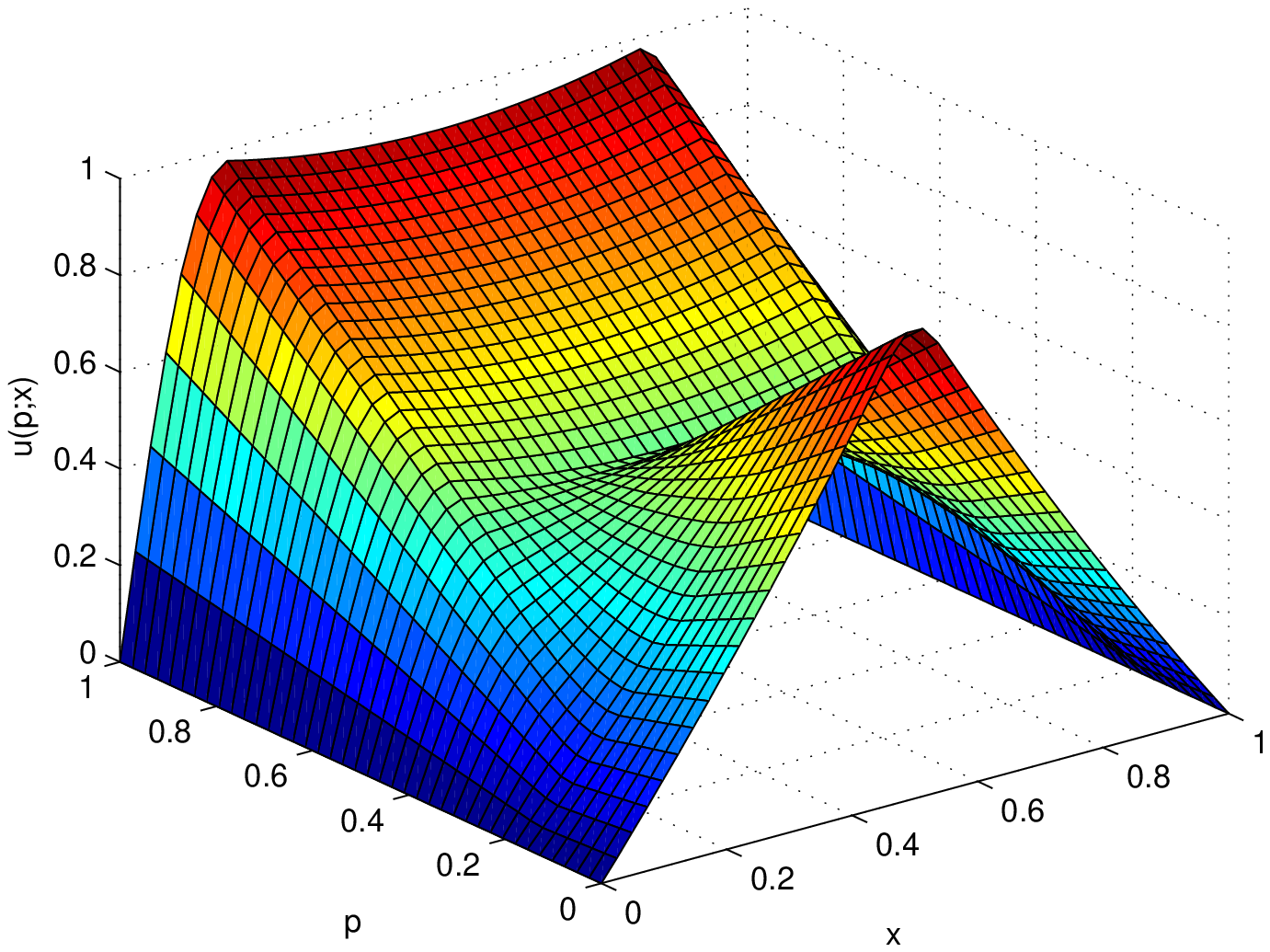}
    \caption{Solution: $u(p;x)$.}
    \label{fig:altitude_solution}
  \end{subfigure}
  \caption{Obstacle and solution as functions of $x$ and $p$ \cite{Cances2011}.}
  \label{fig:altitude_rope}
\end{figure}
The domain $\Omega$ is discretized with $40$ P1 finite elements, while
we use piecewise polynomials of degree 1 on $\cP$. 
The reference solution denoted by $u_{L^2}$ is computed via a
$L^2$-projection, where a BFGS method has been applied at each
quadrature point.



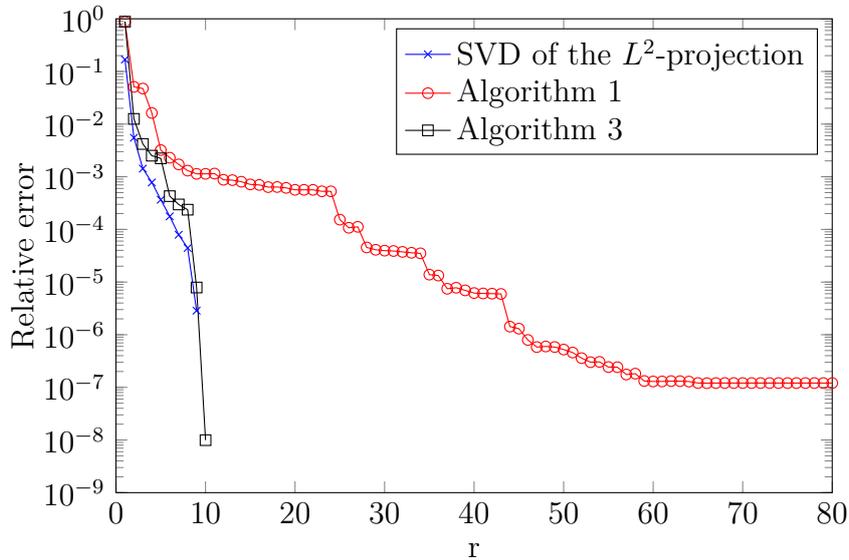
\begin{figure}[h]
  \centering
  \setlength{\figurewidth}{0.6\textwidth}
  \setlength{\figureheight}{0.4\textwidth}
%
%
\begin{tikzpicture}

\begin{axis}[%
width=\figurewidth,
height=\figureheight,
scale only axis,
xmin=0,
xmax=80,
xlabel={r},
ymode=log,
ymin=1e-09,
ymax=1,
yminorticks=true,
ylabel={Relative error},
legend style={draw=black,fill=white,legend cell align=left}
]
\addplot [color=blue,solid,mark=x,mark options={solid}]
  table[row sep=crcr]{
1	0.168518329811575	\\
2	0.00551386958425776	\\
3	0.00143008470234452	\\
4	0.000779778489331186	\\
5	0.000366196099551255	\\
6	0.000176902393123153	\\
7	7.91268433545271e-05	\\
8	4.43472539264875e-05	\\
9	2.85985855905648e-06	\\
10	0	\\
11	0	\\
12	0	\\
13	0	\\
14	0	\\
15	0	\\
16	0	\\
17	0	\\
18	0	\\
19	0	\\
20	0	\\
21	0	\\
22	0	\\
23	0	\\
24	0	\\
25	0	\\
26	0	\\
27	0	\\
28	0	\\
29	0	\\
30	0	\\
31	0	\\
32	0	\\
33	0	\\
34	0	\\
35	0	\\
36	0	\\
37	0	\\
38	0	\\
39	0	\\
};
\addlegendentry{SVD of the $L^2$-projection};

\addplot [color=red,solid,mark=o,mark options={solid}]
  table[row sep=crcr]{
1	0.882720635263839	\\
2	0.0509801525759161	\\
3	0.0473328818445309	\\
4	0.0162640514112494	\\
5	0.00321005187512831	\\
6	0.00230021685987786	\\
7	0.00171686476545333	\\
8	0.00130737280027339	\\
9	0.00113598150061694	\\
10	0.00113590195647984	\\
11	0.0011489533834088	\\
12	0.000876170592959449	\\
13	0.000862631882949828	\\
14	0.000804536989388703	\\
15	0.000715095669513579	\\
16	0.000699764634548391	\\
17	0.000635868229607089	\\
18	0.00063840381331367	\\
19	0.000615432239514108	\\
20	0.000564503034757394	\\
21	0.000564420559155581	\\
22	0.000564220900229845	\\
23	0.000528915181508766	\\
24	0.000528541588976925	\\
25	0.00015328862270715	\\
26	0.000107424031238404	\\
27	0.000111888223672965	\\
28	4.54530732167818e-05	\\
29	4.10481523941911e-05	\\
30	3.9521546807494e-05	\\
31	3.88954132269161e-05	\\
32	3.74123324266966e-05	\\
33	3.59886372514115e-05	\\
34	3.51187639989457e-05	\\
35	1.38379717000368e-05	\\
36	1.3286450791263e-05	\\
37	7.46236483439065e-06	\\
38	7.79966246555929e-06	\\
39	6.94915455085239e-06	\\
40	6.18069429462813e-06	\\
41	6.0753263777186e-06	\\
42	6.02236056754312e-06	\\
43	5.9388808859595e-06	\\
44	1.42574733733029e-06	\\
45	1.30360283030281e-06	\\
46	7.95345936044885e-07	\\
47	5.82256999967796e-07	\\
48	5.9774137785483e-07	\\
49	5.83632906961962e-07	\\
50	5.22615367259405e-07	\\
51	4.58829864330646e-07	\\
52	3.5924640524418e-07	\\
53	3.00638760301625e-07	\\
54	3.05107138253063e-07	\\
55	2.43075348188015e-07	\\
56	2.41571880578832e-07	\\
57	1.75376498448881e-07	\\
58	1.82159467043088e-07	\\
59	1.32068145493578e-07	\\
60	1.29339771544445e-07	\\
61	1.29464941672784e-07	\\
62	1.31420097751201e-07	\\
63	1.31128347429716e-07	\\
64	1.28790081269033e-07	\\
65	1.20878856081853e-07	\\
66	1.20043188547799e-07	\\
67	1.20691050934048e-07	\\
68	1.20749989203841e-07	\\
69	1.20731211713651e-07	\\
70	1.20761417293315e-07	\\
71	1.20746445157038e-07	\\
72	1.20745760962393e-07	\\
73	1.20693604314668e-07	\\
74	1.20697775550975e-07	\\
75	1.20628088231997e-07	\\
76	1.20641491272011e-07	\\
77	1.20634493324691e-07	\\
78	1.20638865228048e-07	\\
79	1.20637168768024e-07	\\
80	1.20635088570634e-07	\\
};
\addlegendentry{Algorithm \ref{alg:basic_pgd}};

\addplot [color=black,solid,mark=square,mark options={solid}]
  table[row sep=crcr]{
1	0.882720635263839	\\
2	0.012555548066394	\\
3	0.00419164611953392	\\
4	0.00251274693012477	\\
5	0.00223769445237738	\\
6	0.000428296605553534	\\
7	0.000297698796209564	\\
8	0.000236824343855339	\\
9	7.87863661841966e-06	\\
10	9.9371094928662e-09	\\
};
\addlegendentry{Algorithm \ref{algo:ni_aama}};

\end{axis}
\end{tikzpicture}%
  \caption{Relative error with respect to the rank of the
    approximation for different algorithms.}
  \label{fig:obs_l2_cvg_rank}
\end{figure}

We use the relative error introduced in Equation \eqref{eq:err_estim}
as an error estimate, where the solution $u$ has been replaced with
its $L^2$-projection $u_{L^2}$. In Figure \ref{fig:obs_l2_cvg_rank},
the relative error with respect to the rank is shown. Three
approximations are illustrated, the truncated SVD of the
$L^2$-projection, the basic PGD (Algorithm \ref{alg:basic_pgd}), and
the improved algorithm (Algorithm \ref{algo:ni_aama}). We used the
same parameters as in Section \ref{sec:electronic-network} for the
different algorithms, except for $C_0$ which is now constructed based
on $(\nabla J(u_r(0);0))^{-1}$. It is directly used in the basic PGD
case, and a block-diagonal version is constructed for the improved PGD
algorithm.

The basic PGD was the technique considered in
Cancès et al. \cite{Cances2011}, asserting in Section 6.2 that
``\emph{this procedure is intrusive in general}'', while we have shown
in this work that we can use a Galerkin approach in a non-intrusive
fashion for constructing a low-rank approximation of the solution.
Concerning the stopping criterion, we used the same as the one used in Section \ref{sec:electronic-network}.

The SVD supplies the best rank $r$ approximation of the reference
solution with respect to the canonical norm $\normt{\cdot}_{\cQ
  \otimes \cU}$. The basic PGD seems to slowly converge
toward the solution of the problem with respect to the canonical norm,
while the improved PGD has a similar convergence as
the truncated SVD, and finally yields a relative error of $10^{-8}$
with a rank 10 approximation.

\begin{figure}[h]
  \centering
  \setlength{\figurewidth}{0.6\textwidth}
  \setlength{\figureheight}{0.4\textwidth}
%
%
\begin{tikzpicture}

\begin{axis}[%
width=\figurewidth,
height=\figureheight,
scale only axis,
xmode=log,
xmin=10000,
xmax=1000000,
xminorticks=true,
xlabel={Residual calls},
ymode=log,
ymin=1e-09,
ymax=1,
yminorticks=true,
ylabel={Relative error},
legend style={draw=black,fill=white,legend cell align=left}
]
\addplot [color=red,solid,mark=o,mark options={solid}]
  table[row sep=crcr]{
49160	0.882720635263839	\\
69040	0.0509801525759161	\\
81560	0.0473328818445309	\\
91480	0.0162640514112494	\\
99320	0.00321005187512831	\\
108720	0.00230021685987786	\\
126600	0.00171686476545333	\\
142280	0.00130737280027339	\\
154760	0.00113598150061694	\\
174400	0.00113590195647984	\\
182560	0.0011489533834088	\\
202920	0.000876170592959449	\\
217400	0.000862631882949828	\\
230280	0.000804536989388703	\\
235880	0.000715095669513579	\\
250560	0.000699764634548391	\\
268120	0.000635868229607089	\\
280320	0.00063840381331367	\\
291240	0.000615432239514108	\\
299560	0.000564503034757394	\\
310920	0.000564420559155581	\\
321320	0.000564220900229845	\\
325360	0.000528915181508766	\\
341120	0.000528541588976925	\\
345920	0.00015328862270715	\\
358320	0.000107424031238404	\\
367920	0.000111888223672965	\\
381800	4.54530732167818e-05	\\
390040	4.10481523941911e-05	\\
403200	3.9521546807494e-05	\\
414960	3.88954132269161e-05	\\
424960	3.74123324266966e-05	\\
431600	3.59886372514115e-05	\\
438960	3.51187639989457e-05	\\
451720	1.38379717000368e-05	\\
464600	1.3286450791263e-05	\\
475080	7.46236483439065e-06	\\
485160	7.79966246555929e-06	\\
499040	6.94915455085239e-06	\\
501800	6.18069429462813e-06	\\
508600	6.0753263777186e-06	\\
515000	6.02236056754312e-06	\\
525480	5.9388808859595e-06	\\
536680	1.42574733733029e-06	\\
544880	1.30360283030281e-06	\\
548440	7.95345936044885e-07	\\
555040	5.82256999967796e-07	\\
561680	5.9774137785483e-07	\\
567960	5.83632906961962e-07	\\
570400	5.22615367259405e-07	\\
572280	4.58829864330646e-07	\\
578040	3.5924640524418e-07	\\
581360	3.00638760301625e-07	\\
582480	3.05107138253063e-07	\\
588400	2.43075348188015e-07	\\
591520	2.41571880578832e-07	\\
596200	1.75376498448881e-07	\\
599120	1.82159467043088e-07	\\
603120	1.32068145493578e-07	\\
604000	1.29339771544445e-07	\\
605480	1.29464941672784e-07	\\
607880	1.31420097751201e-07	\\
608800	1.31128347429716e-07	\\
610360	1.28790081269033e-07	\\
612080	1.20878856081853e-07	\\
614280	1.20043188547799e-07	\\
615240	1.20691050934048e-07	\\
616200	1.20749989203841e-07	\\
617080	1.20731211713651e-07	\\
617960	1.20761417293315e-07	\\
618800	1.20746445157038e-07	\\
619800	1.20745760962393e-07	\\
621480	1.20693604314668e-07	\\
622400	1.20697775550975e-07	\\
623800	1.20628088231997e-07	\\
624680	1.20641491272011e-07	\\
625480	1.20634493324691e-07	\\
626600	1.20638865228048e-07	\\
627520	1.20637168768024e-07	\\
628520	1.20635088570634e-07	\\
};
\addlegendentry{Algorithm \ref{alg:basic_pgd}};

\addplot [color=black,solid,mark=square,mark options={solid}]
  table[row sep=crcr]{
49160	0.882720635263839	\\
89520	0.012555548066394	\\
136960	0.00419164611953392	\\
197920	0.00251274693012477	\\
274160	0.00223769445237738	\\
336840	0.000428296605553534	\\
415920	0.000297698796209564	\\
500600	0.000236824343855339	\\
524760	7.87863661841966e-06	\\
529360	9.9371094928662e-09	\\
};
\addlegendentry{Algorithm \ref{algo:ni_aama}};

\end{axis}
\end{tikzpicture}%
  \caption{Relative error with respect to the
    number of calls to the residual for different algorithms.}
  \label{fig:obs_l2_cvg_time}
\end{figure}
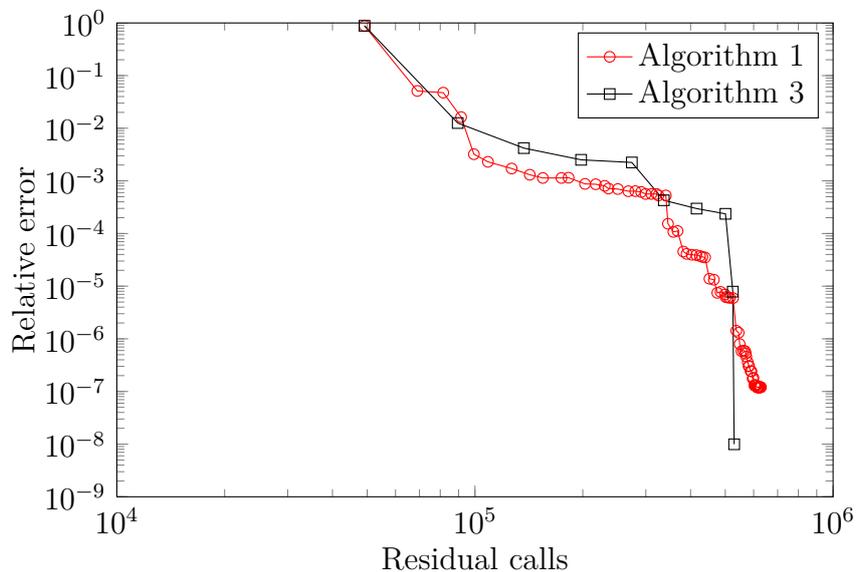

The number of calls to the residual is shown in Figure
\ref{fig:obs_l2_cvg_time} for Algorithms \ref{alg:basic_pgd} and
\ref{algo:ni_aama}. We observe similar speeds of convergence for both
algorithms, while the intrusive implementation of Algorithm
\ref{algo:ni_aama} is usually slower when $J$ is quadratric as it has
been observed in \cite{Giraldi2012}. Note that in this case, the
  basic PGD requires the construction of the computation of 80
  deterministic Hessians, to be compared to the 10 deterministic
  Hessians required by the improved PGD. Moreover, the
$L^2$-projection of the solution only requires 3304 calls to the
residual, which is much less than the number of calls required for the rank one
approximation of the solution.

This example illustrates that accurate low rank approximation of the
solution of nonlinear problems can be directly obtained in a non-intrusive fashion. However, the proposed construction is clearly not
efficient regarding the computational complexity and new
algorithms that are adapted to this non-intrusive setting are clearly required.

\section{Conclusion}
\label{sec:conclusion}

In this work, algorithms for the non-intrusive computation of low-rank
approximations of the solution of a nonlinear stochastic-parametric
problem associated with the minimization of a convex functional have
been proposed. The proposed approach relies on an alternating
minimization algorithm and the BFGS method for minimizing the partial
maps. The techniques are implemented in a non-intrusive way thanks to
the use of numerical integration, and only require the evaluations of
the standard residuum at some parameter values. The method has finally
been applied to two model examples.

The goal of this paper was to revisit standard techniques for
computing a low-rank approximation of the solution of a parametric
equation. The novelty lies in the non-intrusive implementation of
these algorithms. It results that the performance of the classical
methods for computing the low-rank approximation is modified due to the
expensive evaluation of samples (evaluations) of the residual for the numerical
integration.

In order to reduce computational complexity, the number of iterations
of the solver should be reduced to the minimum using more efficient
algorithms than an alternating minimization method plus a BFGS
algorithm. Indeed, one iteration of the solver corresponds to one
integration of the residual. Thus, at each iteration and for each
quadrature point, one sample of the residual must be evaluated.

Also, an adapted integration method could be used in order to further
reduce the number of evaluations of the residual for high dimensional
parametric problems. A first solution is to use an adaptive sparse
grid quadrature technique \cite{Bungartz2004} in order to reduce the
number of quadrature points. Another solution is to build an
approximation of the residual at each iteration. The costly
integration of the residual could then be replaced by a cheaper
approximation of this residual exploiting structured approximation
techniques.

This work is a first step toward Galerkin-based methods for the
low-rank approximation of the solution of parametric nonlinear
equations in a non-intrusive manner, offering new opportunities for
the development of efficient non-intrusive solvers.

\appendix

\section{Proof of Theorem \ref{th:param_framework}}
\label{sec:proof}

\paragraph{Existence and uniqueness of the solution.}

Given that $v\mapsto J(v;p)$ is strongly convex for all $p$ and continuous, there
exists a unique solution to \eqref{eq:param_prob}. As a consequence, we can
define the solution map by $u:p\mapsto u(p) = \argmin{v \in \cU}
J(v;p)$.

\paragraph{Regularity of the solution.}

$v\mapsto J(v;p)$ is Fréchet differentiable and strongly convex
uniformly in $p$. As a
consequence, there exists $\alpha > 0$ independent of $p$ such that
\begin{equation*}
  \inner{\nabla J(u(p);p) - \nabla J(v;p)}{u(p)-v}_\cU \ge \alpha
  \norm{u(p)-v}_\cU^2,\quad \forall p \in \cP,\ \forall v \in \cU.
\end{equation*}
Given that $\nabla J(u(p);p) = 0$, and taking $v=0$ we obtain
\begin{equation*}
  - \inner{\nabla J(0;p)}{u(p)}_\cU = -\inner{A(0;p)-b(p)}{u(p)}_\cU \ge \alpha \norm{u(p)}_\cU^2.
\end{equation*}
The Cauchy-Schwarz inequality gives
\begin{equation*}
  \norm{u(p)}^2_\cU \le \frac{1}{\alpha} \norm{A(0;p)-b(p)}_\cU \norm{u(p)}_\cU,
\end{equation*}
then
\begin{equation*}
  \norm{u(p)}^2_\cU \le \frac{1}{\alpha^2} \norm{A(0;p)-b(p)}_\cU^2
\end{equation*}
which yields $u\in L^2(\cP;\cU)$ using (c). In the following, the set $L^2(\cP;\cU)$ is
identified to the Hilbert space $\cQ\otimes\cU$ equipped with the induced product norm.

\paragraph{Characterization of the solution.}
\label{sec:appen_deriv}
Given that $p\mapsto J(v;p)$ is $\mu$-integrable, we can define the
functional $J_\cP:\cQ\otimes \cU \rightarrow \bR$ such
that
\begin{equation*}
  J_\cP(u) = \int_\cP J(u(p);p) \mu(\di p).
\end{equation*}

\begin{lemma}\label{lem:convex}
  $J_\cP$ is strongly convex.
\end{lemma}
\begin{proof} According to the assumptions of Theorem \ref{th:param_framework}, $v\mapsto J(v;p)$ is strongly convex uniformly in $p$.
  We conclude that, for all $v,w$ in $\cQ\otimes \cU$,
  \begin{align*}
    J_\cP(tv+(1-t)w) &= \int_\cP J(tv(p)+(1-t)w(p);p)
    \mu(\di p) \\
    &\le \int_\cP \left(tJ(v(p);p)+(1-t)J(w(p);p) -\frac{\alpha}{2} t(1-t)
    \norm{v(p)-w(p)}_\cU^2\right) \mu(\di p) \\
    &\le t J_\cP(v)+(1-t)J_\cP(w) - \frac{\alpha}{2} t(1-t)
    \norm{v-w}_{L^2(\cP;\cU)}^2,
  \end{align*}
  and $J_\cP$ is strongly convex.
\end{proof}
\begin{lemma}
  $J_\cP$ is Gâteaux differentiable with Gâteaux derivative
  $\delta J_\cP(u)(\delta u)$ at $u\in L^2(\cP;\cU) = \cQ \otimes \cU$ in
  the direction $\delta u \in \cQ \otimes \cU$ given by
  \begin{equation*}
    \delta J_\cP(u)(\delta u) = \int_{\cP}\inner{\nabla J(u(p);p)}{\delta
      u(p)}_\cU \mu(\di p).
  \end{equation*}
\end{lemma}
\begin{proof}
  Let $t \in (-1,1)\setminus\{0\}$ and $f_t(p)$ being defined by
  $$f_t(p)=\frac{1}{t} \left(J(u(p)+t\delta u(p);p) - J(u(p);p)\right).$$
  Thanks to the fundamental theorem of calculus and Cauchy-Schwarz
  inequality we have
  \begin{align*}
    f_t(p) &=
    \int_0^1 \inner{\nabla J(u(p)+\beta t \delta u(p);p)}{\delta
      u(p)}_\cU \di\beta \\
    &\le \int_0^1 \norm{\nabla J(u(p)+\beta t \delta u(p);p)}_\cU \norm{\delta
      u(p)}_\cU \di\beta.
  \end{align*}
  Given that $v\mapsto \nabla J(v;p) = A(v;p)-b(p)$ is Lipschitz on
  bounded set $\cS$ uniformly in $p$, with $\cS$ being the ball
  centered in $0$ of radius $\normt{u(p)+\delta u(p)}_\cU$, there exists $K>0$ such that
  \begin{align*}
    \norm{\nabla J(u(p)+\beta t \delta u(p);p)}_\cU-\norm{\nabla
      J(0;p)}_\cU &\le \norm{\nabla J(u(p)+\beta t \delta
      u(p);p)-\nabla J(0;p)}_\cU \\
    &\le K \norm{u(p)+\beta t \delta u(p)}_\cU  \\
    &\le K (\normt{u(p)}_\cU + \norm{\delta u(p)}_\cU),
  \end{align*}
  and finally
  \begin{align*}
    f_t(p) &\le
    \int_0^1 (K(\norm{u(p)}_\cU + \norm{\delta u(p)}_\cU)+\norm{\nabla
      J(0;p)}_\cU) \norm{\delta u(p)}_\cU \di \beta \\
    &\le \left(K(\norm{u(p)}_\cU + \norm{\delta u(p)}_\cU) + \norm{A(0;p)-b(p)}_\cU\right)
    \norm{\delta u(p)}_\cU = g(p).
  \end{align*}
  With $u$, $\delta u$ and $p\mapsto
  A(0;p)-b(p)$ being in $L^2(\cP;\cU)$, $g$ is in $L^1(\cP;\cU)$. We
  can thus applied the dominated convergence theorem and state that
  the limit $\delta J_\cP(u)(\delta u)=\lim_{t\rightarrow 0} \int_\cP
  f_t(p) \mu(\di p)$ exists and that
  \begin{align*}
    \delta J_\cP(u)(\delta u) = \int_{\cP}\inner{\nabla J(u(p);p)}{\delta
      u(p)}_\cU \mu(\di p).
  \end{align*}
\end{proof}
\begin{lemma}\label{lem:frechet}
  $J_\cP$ is Fréchet differentiable and, for all $u$, $\delta u$ in
  $L^2(\cP;\cU)$, we have
  \begin{equation*}
    \inner{\nabla J_\cP(u)}{\delta u}_{L^2(\cP;\cU)} = \int_{\cP}\inner{\nabla J(u(p);p)}{\delta
      u(p)}_\cU \mu(\di p).
  \end{equation*}
\end{lemma}
\begin{proof}
  The application $\delta J_\cP:L^2(\cP;\cU)\ni u\mapsto \delta J_\cP(u)\in L^2(\cP;\cU)^*$ is
  linear. Let $u$ and $\delta u$ in $L^2(\cP;\cU)$. Let $\cS \subset
  \cQ\otimes \cU$, be a bounded set containing $u$ and $u+\delta u$. Let $v\in L^2(\cP;\cU)$. There exists
  $K>0$ such that
  \begin{align*}
    \left|\left(\delta J_\cP(u+\delta u)-\delta J_\cP(u)\right)(v)\right| &= \left|\int_\cP
    \inner{\nabla J(u(p)+\delta u(p);p)-\nabla J(u(p);p)}{v(p)}_{\cU}
    \mu(\di p)\right|\\
  &\le \int_\cP \norm{\nabla J(u(p)+\delta u(p);p)-\nabla J(u(p);p)}_\cU
  \norm{v(p)}_\cU\mu(\di p)\\
  &\le \int_\cP K \norm{\delta u(p)}_\cU \norm{v(p)}_\cU\mu(\di p)\\
  &\le K\norm{\delta u}_{L^2(\cP;\cU)} \norm{v}_{L^2(\cP;\cU)}.
  \end{align*}
  It follows that $\delta J_\cP:L^2(\cP;\cU) \rightarrow
  L^2(\cP;\cU)^*$ is continuous. As a consequence, $J_\cP$ is
  Fréchet differentiable with Fréchet derivative
  \begin{equation*}
     \inner{\nabla J_\cP(u)}{\delta u}_{L^2(\cP;\cU)} =
     \delta J_\cP(u)(\delta u)= \int_{\cP}\inner{\nabla J(u(p);p)}{\delta
      u(p)}_\cU \mu(\di p).
  \end{equation*}
\end{proof}
According to Lemma \ref{lem:convex}, there exists an unique minimizer to
$J_\cP$. The combination of Lemmas \ref{lem:convex} and
\ref{lem:frechet} ensures that the minimizer is equivalently
characterized by
\begin{equation*}
  \inner{\nabla J_\cP(u)}{\delta u}_{L^2(\cP;\cU)} = \int_{\cP}\inner{A(u(p);p)-b(p)}{\delta
      u(p)}_\cU \mu(\di p) = 0,\quad \forall \delta u \in L^2(\cP;\cU).
\end{equation*}
Given that the map $u:p\mapsto \arg\min_{v\in\cU} J(v;p)$ is in
$L^2(\cP;\cU)$ and satisfies $A(u(p);p)-b(p)=0$ for all $p$, it is the
unique minimizer of $J_\cP$.

\section{Proof of Theorem  \ref{th:min_altern}}
\label{sec:proof-theorem-min-alt}

\paragraph{Existence and uniqueness of the solutions.}
\label{sec:exist-uniq-solut}
According to Lemma \ref{lem:convex}, $J_\cP$ is
strongly convex. Let $J_\cP^{\vek{v}}$ and
  $J_\cP^{\vek{\lambda}}$ denote $\vek{\lambda} \mapsto J_\cP^{\vek{v}}(\vek{\lambda})=J_\cP \circ
F_r(\vek{\lambda},\vek{v})$ and $\vek{v} \mapsto J_\cP^{\vek{\lambda}}(\vek{v})=J_\cP \circ
F_r(\vek{\lambda},\vek{v})$. For $t \in (0,1)$, given that $F_r$ is bilinear (Lemma \ref{lem:bilin_conti}) we have
\begin{align*}
  J_{\cP}^{\vek{\lambda}}(t \vek{v} + (1-t) \vek{w}) &= J_\cP(t
  F_r(\vek{\lambda},\vek{v}) +(1-t) F_r(\vek{\lambda},\vek{w})) \\
  & \le t J_\cP^{\vek{\lambda}}(\vek{v}) + (1-t)
  J_\cP^{\vek{\lambda}}(\vek{w}) - \frac{\alpha}{2} t(1-t)
  \norm{F_r(\vek{\lambda},\vek{v}-\vek{w})}_{\cQ \otimes \cU}^2.
\end{align*}
We have
\begin{align*}
  \norm{F_r(\vek{\lambda},\vek{v}-\vek{w})}_{\cQ \otimes \cU}^2 &=
  \sum_{i=1}^r \sum_{j=1}^r \inner{\lambda_i}{\lambda_j}_\cQ
  \inner{v_i-w_i}{v_j-w_j}_\cU \\ &= \sum_{i=1}^r
  \inner{v_i-w_i}{\sum_{j=1}^r\inner{\lambda_i}{\lambda_j}_\cQ
    (v_j-w_j) }_\cU \\ &= \inner{\vek{v}-\vek{w}}{G_{\vek{\lambda}}(\vek{v}-\vek{w})}_{\cU^r},
\end{align*}
$G_{\vek{\lambda}} \in \bR^{r\times r}$ being the Gram matrix associated to
$\vek{\lambda}$ and its application to $\vek{v}-\vek{w}$ being detailed in Section
\ref{sec:low-rank-appr}. If $\vek{\lambda}$ is a set of linearly
independent functions, then $G_{\vek{\lambda}}$ is symmetric positive
definite. It defines then an induced norm such that
$\normt{F_r(\vek{\lambda},\vek{v}-\vek{w})}_{\cQ \otimes
  \cU}^2=\normt{\vek{v}-\vek{w}}_{G_{\vek{\lambda}}}^2$ yielding that
$J_\cP^{\vek{\lambda}}$ is strongly convex. Note that when
$\vek{\lambda}$ is an orthonormal set, we have $\normt{F_r(\vek{\lambda},\vek{v}-\vek{w})}_{\cQ \otimes
  \cU}^2=\normt{\vek{v}-\vek{w}}_{\cU^r}^2$ and
$J_\cP^{\vek{\lambda}}$ has the same convexity constant than $J(\cdot;p)$.
Similarly, if $\vek{v}$ is a set of linearly independent vectors, $J_\cP^{\vek{v}}$ is strongly convex. As a consequence, there exists a
unique solution to the minimization problems
\begin{equation*}
  \min_{\vek{\lambda}\in\cQ^r} J_\cP^{\vek{v}}(\vek{\lambda})
  \quad \text{and} \quad \min_{\vek{v}\in\cU^r}  J_\cP^{\vek{\lambda}}(\vek{v}).
\end{equation*}

\paragraph{Characterization of the solutions.}
\label{sec:stat-cond-1}

$J_\cP$ and $F_r$ being Fréchet differentiable, $J_\cP^{\vek{v}}$ and
$J_\cP^{\vek{\lambda}}$ are Fréchet differentiable. Given that
$J_\cP^{\vek{v}}$ and $J_\cP^{\vek{\lambda}}$ are strongly convex, we
know that the solution to the minimization problems are uniquely
characterized by the equations
\begin{equation*}
  \inner{\nabla J_\cP^{\vek{v}}(\vek{\lambda})}{\delta \vek{\lambda}}_{\cQ^r} = 0,
  \quad \forall \delta \vek{\lambda} \in \cQ^r \qquad
  \text{and} \qquad \inner{\nabla J_\cP^{\vek{\lambda}}(\vek{v})}{\delta \vek{v}}_{\cU^r} = 0,
  \quad \forall \delta \vek{v} \in \cU^r.
\end{equation*}
In Appendix \ref{sec:appen_deriv}, we established that 
\begin{align*}
  \inner{\nabla J_\cP(u)}{\delta u}_{L^2(\cP;\cU)} =
  \int_\cP\inner{A(u(p);p)-b(p)}{\delta u(p)}_\cU \mu(\di p), \quad
  \forall u,\delta u \in \cQ \otimes \cU.
\end{align*}
Using the chain rule, we find that 
\begin{align*}
  \inner{\nabla J_\cP^{\vek{v}}(\vek{\lambda})}{\delta
    \vek{\lambda}}_{\cQ^r} &= \int_\cP
  \inner{A(F_r(\vek{\lambda},\vek{v})(p);p)-b(p)}{F_r(\delta
    \vek{\lambda},\vek{v})(p)}_\cU \mu(\di p) \\
  \text{and} \quad \inner{\nabla J_\cP^{\vek{\lambda}}(\vek{v})}{\delta
    \vek{v}}_{\cU^r} &= \int_\cP
  \inner{A(F_r(\vek{\lambda},\vek{v})(p);p)-b(p)}{F_r(\vek{\lambda},\delta
    \vek{v})(p)}_\cU \mu(\di p).
\end{align*}

\section{Proof of Proposition \ref{prop:elec_network}}
\label{sec:proof_prop_elec_network}
\begin{enumerate}[(a)]
  \item Since $\mathbf{B}$ is
    symmetric positive definite, the functional $J$ is given by
    $J(v;p) = \frac{1}{2}
    v^T B v +
    \frac{1}{4}(p_1+2) (v^T v)^2 - (p_2+25) v^Tf$.
  \item $v\mapsto J(v;p)$ is clearly Fréchet differentiable. Moreover,
    the Hessian of $v\mapsto J(v;p)$ is given by $v\mapsto H(v;p) = B
    + (p_1+2)(v^Tv I + v\otimes v)$. We have $\delta v^T
    H(v;p) \delta v \ge \delta v^T B \delta v \ge \alpha \normt{\delta
      v}_\cU^2$, $\alpha$ being the smallest eigenvalue of $B$ which
    is independent of $p$.
  \item $p \mapsto -(p_2+25) f$ is integrable on $\cP$.
  \item $p\mapsto J(v;p)$ is clearly integrable on $\cP$. Concerning
    the Lipschitz continuity property, given a bounded set $\cS \subset
  \cU$, and given that $p_1 \le 1$, we have
  \begin{align*}
    \norm{\nabla J(v;p)-\nabla J(w;p)}_\cU = \norm{B(v-w) + (p_1+2) (\norm{v}_\cU^2 v -
    \norm{w}_\cU^2 w)}_\cU \\
  \le \norm{B} \norm{v-w}_\cU + 3 \norm{\norm{v}_\cU^2 v - \norm{w}_\cU^2 w}_\cU,
  \end{align*}
  for all $v,w \in \cS$, where $\norm{B}$ is the
  operator norm. We introduce the operator $C = (v\otimes v + v \otimes w + w
  \otimes v + w \otimes w)$. It is bounded since
  $\cS$ is bounded and we can define $\norm{C}$. We notice then
  \begin{align*}
    \norm{\norm{v}_\cU^2 v - \norm{w}_\cU^2 w}_\cU &=  \norm{C(v-w) + \norm{v}_\cU^2
    (v-w) - (\norm{v}_\cU - \norm{w}_\cU)(\norm{v}_\cU+\norm{w}_\cU)
    v}_\cU,\\
  &\le \left(\norm{C}+\norm{v}_\cU^2\right)\norm{v-w}_\cU  + |\norm{v}_\cU -
  \norm{w}_\cU|(\norm{v}_\cU+\norm{w}_\cU)\norm{v}_\cU, \\
  &\le \left(\norm{C}+\norm{v}_\cU^2 +
    (\norm{v}_\cU+\norm{w}_\cU)\norm{v}_\cU \right)\norm{v-w}_\cU,
\end{align*}
  using the inequality $|\norm{v}_\cU - \norm{w}_\cU| \le
  \norm{v-w}_\cU$. For a given bounded set $\cS \subset \cU$, we
  denote by
  $D = \sup_{v \in \cS} \norm{v}_\cU$. With $K = \norm{B}+ 3 \norm{C} + 9 D^2 > 0$, we
  finally deduce
  \begin{align*}
    \norm{\nabla J(v;p)-\nabla J(w;p)}_\cU \le K\norm{v-w}_\cU, \qquad \forall
    v,w \in \cS,
  \end{align*}
  with $K$ independent of $p$.
\end{enumerate}

\bibliographystyle{abbrv}

\end{document}